\documentclass[10pt,bibliography=totoc]{article}
\usepackage{etex}
\usepackage[T1]{fontenc} 
\usepackage{latexsym}
\usepackage{empheq}
\usepackage{amsmath,amssymb,amsfonts,amsthm}
\usepackage{amscd} 
\usepackage{makecell}

\usepackage[american]{babel}
\usepackage{eqnarray}
\usepackage{graphicx}
\usepackage{tabularx}
\usepackage{caption}
\usepackage[table]{xcolor}
\usepackage{colortbl}
\usepackage{multicol}
\usepackage{multirow}
\usepackage{tikz}
\usetikzlibrary{intersections,shapes.arrows}
\usetikzlibrary{decorations.pathreplacing,decorations.markings}
\usetikzlibrary{arrows,decorations.markings}
\usepackage{subfigure}
\usepackage{epsfig}
\usepackage{textcomp}
\usepackage{color,xcolor}
\usepackage{bbm} 
\usepackage{enumitem} 
\usepackage{marginnote} 
\usepackage{todonotes}
\presetkeys{todonotes}{size=\footnotesize, color=white, linecolor = black}{}
\usepackage{mathtools}
\usepackage{extarrows}
\usepackage{filecontents}
\usepackage{diagbox}
\usepackage{xspace}
\usepackage{microtype}
\usepackage{csquotes} 

\usepackage{esint} 
\usepackage{float}
\usepackage[text={6.25in,9in},centering]{geometry}
\usepackage{authblk}
\usepackage{tensor} 
\usepackage{tikz-cd}

\allowdisplaybreaks

\usepackage[hyphens]{url}
\usepackage[
	backend=biber,
	bibencoding=utf8,
	style=numeric,
	citestyle=numeric,sorting=nyvt,sortcites=false,
	giveninits=true,
	maxnames=10,
	url=true,eprint=false,doi=false,isbn=false,
	hyperref=true,
	backref=false
]{biblatex}
\renewbibmacro{in:}{}
\AtEveryBibitem{\ifentrytype{book}{\clearfield{pages}}{}}
\AtEveryBibitem{\clearfield{note}} 

\usepackage[hyphenbreaks]{breakurl}
\usepackage{cleveref}%

\numberwithin{equation}{section}

\theoremstyle{plain}
  \newtheorem{theorem}{\sffamily Theorem}[section]
  \newtheorem{proposition}[theorem]{\sffamily Proposition}\Crefname{proposition}{Proposition}{Propositions}
  \newtheorem{lemma}[theorem]{\sffamily Lemma}\Crefname{lemma}{Lemma}{Lemmas}
  \newtheorem{corollary}[theorem]{\sffamily Corollary}
   
  \theoremstyle{definition}
  \newtheorem{remark}[theorem]{\sffamily Remark}\Crefname{remark}{Remark}{Remarks}
  \newtheorem{definition}[theorem]{\sffamily Definition}
  
  \newtheorem*{acknowledgment*}{Acknowledgment}
\usepackage{tikz, graphicx}
\usetikzlibrary{plotmarks}

\usepackage{algorithm}
\usepackage{algpseudocode}

\usepackage{tikz}
\usetikzlibrary{intersections,shapes.arrows}
\usetikzlibrary{decorations.pathreplacing,decorations.markings}
\usetikzlibrary{arrows,decorations.markings}
\usetikzlibrary{spy}
\usepackage{pgfplots}
\usepgfplotslibrary{groupplots}

\newcommand{\Ubf}{\ensuremath{\mathbf{U}}}
\newcommand{\Zbf}{\ensuremath{\mathbf{Z}}}
\newcommand{\Acal}{\ensuremath{\mathcal{A}}}
\newcommand{\Bcal}{\ensuremath{\mathcal{B}}}
\newcommand{\Ccal}{\ensuremath{\mathcal{C}}}
\newcommand{\Fcal}{\ensuremath{\mathcal{F}}}
\newcommand{\Gcal}{\ensuremath{\mathcal{G}}}
\newcommand{\Hcal}{\ensuremath{\mathcal{H}}}
\newcommand{\Ical}{\ensuremath{\mathcal{I}}}
\newcommand{\Jcal}{\ensuremath{\mathcal{J}}}
\newcommand{\Lcal}{\ensuremath{\mathcal{L}}}
\newcommand{\Rcal}{\ensuremath{\mathcal{R}}}
\newcommand{\Scal}{\ensuremath{\mathcal{S}}}
\newcommand{\Tcal}{\ensuremath{\mathcal{T}}}

\newcommand{\Vcal}{\ensuremath{\mathcal{V}}}
\newcommand{\Xcal}{\ensuremath{\mathcal{X}}}

\newcommand{\norm}[1]{\|{#1}\|}
\newcommand{\spn}[1]{\mbox{span}\,\!\{{#1}\}}
\newcommand{\Vd}[1]{\ensuremath{\mathcal{V}_{#1}}}
\newcommand{\Ud}{\ensuremath{\mathcal{U}}}
\newcommand{\br}[3]{\ensuremath{\{{#1},{#2}\}_{#3}}}
\newcommand{\Ld}[1]{\ensuremath{C^{\infty}(\Vd{#1})}}
\newcommand{\J}[1]{\ensuremath{J_{#1}}}
\newcommand{\Jc}[1]{\ensuremath{{\Jcal}_{#1}}}
\newcommand{\Ham}{\ensuremath{\Hcal}}
\newcommand{\HamN}{\ensuremath{\Hcal}}
\newcommand{\R}[2]{\ensuremath{\mathbb{R}^{{#1}\times{#2}}}}
\newcommand{\C}[2]{\ensuremath{\mathbb{C}^{{#1}\times{#2}}}}
\renewcommand{\r}[1]{\ensuremath{\mathbb{R}^{#1}}}
\renewcommand{\c}[1]{\ensuremath{\mathbb{C}^{#1}}}
\newcommand{\Id}{\mbox{\,Id\,}}
\renewcommand{\Im}[1]{\ensuremath{\mathrm{Im}({#1})}}
\newcommand{\rank}[1]{\ensuremath{\mathrm{rank}({#1})}}

\newcommand{\Nf}{\ensuremath{{2N}}} 
\newcommand{\Nr}{\ensuremath{{2n}}} 
\newcommand{\Nfh}{\ensuremath{N}} 
\newcommand{\Nrh}{\ensuremath{n}} 
\newcommand{\Np}{\ensuremath{p}} 
\newcommand{\Ns}{\ensuremath{{s}}} 
\newcommand{\Nm}{\ensuremath{{k}}} 
\newcommand{\umi}{\ensuremath{U_m^i}}
\newcommand{\umj}{\ensuremath{U_m^j}} 
\newcommand{\mmi}{\ensuremath{\Omega_m^i}}

\newcommand{\exd}{\mathsf{d}} 
\newcommand{\contr}{\mathsf{i}} 
\newcommand{\dt}{\ensuremath{\Delta t\,}} 

\DeclareMathOperator*{\sym}{Sym}
\DeclareMathOperator*{\St}{St}
\DeclareMathOperator*{\Sp}{Sp}

\DeclareMathOperator*{\Un}{\mathcal{U}}
\DeclareMathOperator*{\spl}{spl}
\DeclareMathOperator*{\opt}{opt}
\DeclareMathOperator*{\tr}{tr}

\newcommand{\bl}[1]{{{\color{black}{#1}}}}

\newcommand{\Mcal}{\ensuremath{\mathcal{M}}}
\newcommand{\TM}[1]{\ensuremath{T_{#1}\Mcal}}
\newcommand{\TsM}[1]{\ensuremath{H_{#1}}}
\newcommand{\T}[2]{\ensuremath{T_{#1}#2}}
\newcommand{\fr}[1]{\ensuremath{\mathfrak{#1}}}
\newcommand{\Idm}{\ensuremath{I}}
\newcommand{\Zrm}{\ensuremath{0}}
\newcommand{\prm}{\ensuremath{\eta}} 
\newcommand{\prmh}{\ensuremath{\eta_h}} 
\newcommand{\Sprm}{\ensuremath{\Gamma}} 
\newcommand{\Sprmh}{\ensuremath{\Sprm_h}} 
\newcommand{\cay}{\ensuremath{\mathrm{cay}}}
\newcommand{\dcay}{\ensuremath{\mathrm{dcay}}}
\renewcommand{\exp}{\ensuremath{\mathrm{exp}}}
\newcommand{\dexp}{\ensuremath{\mathrm{dexp}}}
\newcommand{\coo}{\ensuremath{\psi}}
\newcommand{\dcoo}{\ensuremath{\mathrm{d}\psi}}
\newcommand{\ad}[1]{\ensuremath{\mathrm{ad}_{#1}}}
\newcommand{\dU}{\ensuremath{\dot{U}}}
\newcommand{\dZ}{\ensuremath{\dot{Z}}}
\newcommand{\hop}{\ensuremath{\mathsf{H}}}
\newcommand{\rqv}{\ensuremath{\Rcal_Q(V)}}
\newcommand{\vt}{\ensuremath{\widetilde{V}}}
\newcommand{\aqv}{\ensuremath{\Upsilon_Q(V)}}
\newcommand{\aqvt}{\ensuremath{\Upsilon_Q(\widetilde{V})}}
\newcommand{\aqvb}{\ensuremath{\overline{\Upsilon}_Q(V)}}
\newcommand{\gsv}{\ensuremath{\Theta^S_Q(\vt)}}
\newcommand{\gt}{\ensuremath{\widetilde{\Theta}}}
\newcommand{\dr}{\ensuremath{{\bl{\mathrm{d}}\Rcal}}}
\newcommand{\drq}[1]{\ensuremath{{\bl{\mathrm{d}}\Rcal_Q}_{\big|_{#1}}}}
\title{
Dynamical Reduced Basis Methods
for Hamiltonian Systems}
\author[]{%
Cecilia Pagliantini\thanks{Centre for Analysis, Scientific computing and Applications,
			  Department of Mathematics and Computer Science,
			  Eindhoven University of Technology,
			  The Netherlands.\\
Email: \texttt{c.pagliantini@tue.nl}}
}
\date{November 2019. Updated on May 2021.}

\begin{document}

%
%
%
%
%

\maketitle

\begin{abstract}
We consider model order reduction of parameterized Hamiltonian systems describing nondissipative phenomena, like wave-type and transport dominated problems. The development of reduced basis methods for such models is challenged by two main factors: the rich geometric structure encoding the physical and stability properties of the dynamics and its \emph{local} low-rank nature. To address these aspects, we propose a nonlinear structure-preserving model reduction where the reduced phase space evolves in time. In the spirit of dynamical low-rank approximation, the reduced dynamics is obtained by a symplectic projection of the Hamiltonian vector field onto the tangent space of the approximation manifold at each reduced state. A priori error estimates are established in terms of the projection error of the full model solution onto the reduced manifold. For the temporal discretization of the reduced dynamics we employ splitting techniques. The reduced basis satisfies an evolution equation on the manifold of symplectic and orthogonal rectangular matrices having one dimension equal to the size of the full model. We recast the problem on the tangent space of the matrix manifold and develop intrinsic temporal integrators based on Lie group techniques together with explicit Runge--Kutta (RK) schemes. The resulting methods are shown to converge with the order of the RK integrator and their computational complexity depends only linearly on the dimension of the full model, provided the evaluation of the reduced flow velocity has a comparable cost.
\end{abstract}

\textbf{MSC 2010.} 37N30, 65P10, 15A24, 78M34.

\textbf{Keywords.} Hamiltonian dynamics, symplectic manifolds, dynamical low-rank approximation,
reduced basis methods (RBM), Lie group integrators.

\section{Introduction}
\label{intro}
%
\bl{Hamiltonian mechanics is a cornerstone of physics and has provided the mathematical foundation for the equations of motion of systems that describe conservative processes.
Hamiltonian systems
can be viewed as dynamical extension of the first law of thermodynamics.} 
\bl{In this work,} we consider parameterized finite-dimensional canonical Hamiltonian systems: these can model \bl{energy-conserving}
nondissipative flows 
or can ensue from the numerical discretization of partial differential equations
derived from action principles.
%
%
\bl{
Many relevant models in mathematical physics can be written as Hamiltonian systems, 
and find application in, for example,
classical mechanics,
quantum dynamics,
population and epidemics dynamics.
Furthermore, partial differential equations that can be derived from action principles include Maxwell's equations, Schr\"odinger's equation, Korteweg--de Vries and the wave equation, compressible and incompressible Euler equations, Vlasov--Poisson and Vlasov--Maxwell equations.

Our target problem is as follows.} Let $\Tcal:=(t_0,T]$ be a temporal interval and let
$\Vd{\Nf}$ be a $\Nf$-dimensional 
vector space.
Let $\Sprm\subset \mathbb{R}^d$, with $d\geq 1$,
be a compact set of parameters.
For each $\prm\in\Sprm$,
we consider the initial value problem:
For $u_0(\prm)\in\Vd{\Nf}$, find $u(\cdot,\prm)\in C^1(\Tcal,\Vd{\Nf})$ such that
\begin{equation}\label{eq:dynIntro}
	\left\{
	\begin{array}{ll}
		\partial_t u(t,\prm) = \Xcal_{\HamN}(u(t,\prm),\prm), & \quad \quad\mbox{for }\;t\in\Tcal,\\
		u(t_0,\prm) = u_0(\prm),&
	\end{array}\right.	
\end{equation}
where $\Xcal_{\HamN}(u,\prm)\in \Vd{\Nf}$ is the Hamiltonian vector field at time $t\in\Tcal$,
and $C^1(\Tcal,\Vd{\Nf})$ denotes continuous differentiable functions in time taking values in $\Vd{\Nf}$.
\bl{Numerical simulations of systems like \eqref{eq:dynIntro}
can become prohibitively expensive, in terms of computational cost, if the number $\Nf$ of degrees of freedom is large.
In the context of long-time and many-query simulations,
this often leads to unmanageable demands on computational resources.
Model order reduction aims at alleviating this computational burden
by replacing the original high-dimensional problem with
a low-dimensional, efficient model that is fast to solve but that approximates well the underlying full-order dynamics.}
%
When dealing with Hamiltonian systems additional difficulties are encountered to
ensure that the geometric structure of the phase space, the stability and
the conservation properties of the original system are not hindered
during the reduction.
\bl{The main goal of this work is to develop and analyze structure-preserving model order reduction methods for the efficient, accurate, and physically consistent approximation
of high-dimensional parametric Hamiltonian systems.}

\bl{Within model order reduction techniques, projection-based} reduced basis methods (RBM)
consist in building,
during a computationally intensive offline phase,
a reduced basis from a proper orthogonal decomposition
of a set of high-fidelity simulations (referred to as snapshots)
at sampled values of time and parameters.
A reduced dynamics is then obtained via projection
of the full model onto the lower dimension space spanned by the reduced basis.
Projection-based RBM for Hamiltonian systems
tailored to preserve the geometric structure of the dynamics
were developed in \cite{LKM03} and \cite{CTB15} using
a variational Lagrangian formulation of the problem,
in \cite{PM16,AH17,BBH19} for canonically symplectic dynamical systems,
and in \cite{HP18} to deal with
Hamiltonian problems whose phase space is endowed with a \bl{state-dependent} Poisson manifold structure.
Although the aforementioned approaches can provide robust and efficient reduced models,
they might require a sufficiently large approximation space to achieve even moderate accuracy.
This can be ascribed to the fact that nondissipative phenomena, like advection and wave-type problems,
do not possess a \emph{global} low-rank structure, and are therefore characterized 
by slowly decaying Kolmogorov widths, as highlighted in 
\bl{\cite{deVore17}}. 
Hence, local reduced spaces seem to provide a more effective instrument
to deal with this kind of dynamical systems.

In this work we propose a nonlinear projection-based model order reduction
of parameterized Hamiltonian systems where the reduced basis is dynamically evolving in time.
The idea is to consider a modal decomposition of the approximate solution to \eqref{eq:dynIntro}
of the form
\begin{equation}\label{eq:exp}
	u(t,\prm)\approx 
		\sum_{i=1}^{\Nr} U_i(t) Z_i(t,\prm),\qquad n\ll N,\quad\forall\, t\in\Tcal,\,\prm\in\Sprm,
\end{equation}
where the reduced basis $\{U_i\}_{\bl{i=1}}^{\Nr}\subset\r{\Nf}$, and
the expansion coefficients $\{Z_i\}_{\bl{i=1}}^{\Nr}\subset\r{}$
can both change in time.
The approximate reduced flow is then generated by the velocity field
resulting from the projection of the vector field $\Xcal_{\HamN}$ in \eqref{eq:dynIntro}
into the tangent space of the reduced space at the current state.
By imposing that the evolving reduced space spanned by $\{U_i\}_{\bl{i=1}}^{\Nr}$ is a symplectic manifold at every time
the continuous reduced dynamics preserves the geometric structure of the full model.

Low-rank approximations based on a modal decomposition of the approximate solution
with dynamically evolving modes similar to \eqref{eq:exp},
have been widely studied in quantum mechanics
in the
multiconfiguration time-dependent Hartree (MCTDH) method, see e.g. \cite{Lubich08}.
In the finite dimensional setting,
a similar approach, known as
dynamical low-rank approximation \cite{KL07},
provides
a low-rank factorization updating technique
to efficiently compute approximations of
time-dependent large data matrices,
by projecting the matrix time derivative onto the tangent space
of the low-rank matrix manifold.
For the discretization of time-dependent stochastic PDEs, Sapsis and Lermusiaux
proposed in \cite{SL09} the so-called
dynamically orthogonal (DO) scheme,
where the deterministic approximation space
adapts over time by evolving according to the
differential operator describing the stochastic problem.
A connection between dynamical low-rank approximations and DO methods was established in \cite{MNZ15}.
Further, a geometric perspective on
the relation between dynamical low-rank approximation, DO field equations and model order reduction
in the context of time-dependent matrices
has been investigated in \cite{FL18}.
To the best of our knowledge, the only work to address structure-preserving dynamical low-rank approximations
is \cite{MN17},
where the authors develop a DO discretization of stochastic PDEs possessing a symplectic Hamiltonian structure.
The method proposed in \cite{MN17} consists in recasting the continuous PDE
into the complex setting
and then applying a dynamical low-rank strategy
to derive field equations for the evolution of the stochastic modal decomposition of the approximate solution.
The approach we propose for the nonlinear model order reduction of problem \eqref{eq:dynIntro}
adopts a geometric perspective similar to \cite{FL18} and
yields an evolution equation for the reduced solution
analogous to \cite{MN17}, although we do not resort to
a reformulation of the evolution problem in a complex framework.

Concerning the temporal discretization of the reduced dynamics describing the evolution of the approximate solution \eqref{eq:exp},
the low-dimensional system for the expansion coefficients $\{Z_i\}_{\bl{i=1}}^{\Nr}$
is Hamiltonian and
can be approximated using standard symplectic integrators.
On the other hand, the development of numerical schemes for the evolution of the reduced basis
is more involved as two major challenges need to be addressed:
(i) a structure-preserving approximation requires that
the discrete evolution remains on the manifold of
symplectic and (semi-)orthogonal rectangular matrices;
(ii) since the reduced basis forms a matrix with
one dimension equal to the size of the full model,
the effectiveness of the model reduction might be thwarted
by the computational cost associated with the numerical solution of the corresponding evolution equation.
Various methods have been proposed in the literature to solve differential equations on manifolds, see e.g. \cite[Chapter IV]{HaLuWa06}.
Most notably projection methods apply a conventional discretization scheme and, after each time step,
a ``correction'' is made by projecting the updated approximate solution to the constrained manifold.
Alternatively, methods based on the use of local parameterizations of the manifold,
so-called \emph{intrinsic}, are well-developed in the context of differential equations on Lie groups, \emph{cf.} \cite[Section IV.8]{HaLuWa06}.
The idea is to recast the evolution equation 
in the corresponding Lie algebra, which is a linear space, and to then recover an approximate solution
in the Lie group via
local coordinate maps.
Instrinsic methods possess excellent structure-preserving properties
provided the local coordinate map can be computed exactly.
However,
they usually require a considerable computational cost associated with
the evaluation of the coordinate map and its inverse at every time step (possibly at every stage within each step).

We propose and analyze two structure-preserving temporal approximations
and show that their computational complexity scales linearly with the dimension of the full model,
under the assumption that the velocity field of the reduced flow can be evaluated at
a comparable cost.
\bl{The first algorithm we propose is a Runge--Kutta Munthe--Kaas (RK-MK) method \cite{MK95}, and we rely on the action on the orthosymplectic matrix manifold by the quadratic Lie group of unitary matrices.}
%
By exploiting the structure of our dynamical low-rank approximation and the properties of the
local coordinate map supplied by the Cayley transform,
we prove
the computational efficiency of this algorithm with respect to
the dimension of the high-fidelity model. However,
a polynomial dependence on the number of stages of the RK temporal integrator
might yield high computational costs in the presence of full models of moderate dimension.
To overcome this issue, we propose a discretization scheme based on the use of
retraction maps to recast the local evolution of the reduced basis
on the tangent space of the matrix manifold at the current state, inspired by the works \cite{CO02,CO03}
on intrinsic temporal integrators for orthogonal flows.

The remainder of the paper is organized as follows.
In Section \ref{sec:Hamdyn} the geometric structure underlying the dynamics
of Hamiltonian systems is presented, and the concept of orthosymplectic basis
spanning the approximate phase space is introduced.
In Section \ref{sec:OSmatrix} we describe the properties of linear symplectic maps
needed to guarantee that the geometric structure of the full dynamics
is inherited by the reduced problem.
Subsequently, in Section \ref{sec:SDLR} we develop and analyze a dynamical low-rank approximation strategy
resulting in dynamical systems for the reduced orthosymplectic basis and the corresponding expansion coefficients
in \eqref{eq:exp}.
In Section \ref{sec:EvolROM} efficient and structure-preserving temporal integrators for
the reduced basis evolution problem are derived.
\bl{Section~\ref{sec:numExp} concerns a numerical test where the proposed method is compared to a global reduced basis approach.}
%
We present some concluding remarks and open questions in Section \ref{sec:conclusions}.

\section{Hamiltonian dynamics on symplectic manifolds}\label{sec:Hamdyn}
The phase space of Hamiltonian dynamical systems is endowed with a differential
Poisson manifold structure
which underpins the physical properties of the system.
Most prominently, Poisson structures encode a family of conserved quantities that, by Noether's theorem,
are related to symmetries of the Hamiltonian.
Here we focus on dynamical systems whose phase space
has a global Poisson structure that is canonical and nondegenerate,
namely symplectic.
\bl{\begin{definition}[Symplectic vector space]\label{def:sympform}
Let $\Vd{\Nf}$ be a $\Nf$-dimensional real vector space.
A skew-symmetric bilinear form $\omega:\Vd{\Nf}\times \Vd{\Nf}\rightarrow\mathbb{R}$ is \emph{symplectic} if it is nondegenerate, i.e., if
$\omega(u,v)=0$, for any $v\in\Vd{\Nf}$, then $u=0$.
The map $\omega$ is called a \emph{linear symplectic
structure} on $\Vd{\Nf}$, and $(\Vd{\Nf},\omega)$ is called a \emph{symplectic vector space}.
\end{definition}}
\bl{On a finite $\Nf$-dimensional smooth manifold $\Vd{\Nf}$,
let $\omega$ be a 2-form, that is, for any $p\in\Vd{\Nf}$, the map
$\omega_p:T_p\Vd{\Nf}\times T_p\Vd{\Nf}\rightarrow\mathbb{R}$ is skew-symmetric and bilinear on the tangent space to $\Vd{\Nf}$ at $p$,
and it varies smoothly in $p$.
The 2-form $\omega$ is a \emph{symplectic structure} if it is closed and
$\omega_p$ is symplectic for
all $p\in\Vd{\Nf}$, in the sense of Definition~\ref{def:sympform}.
A manifold $\Vd{\Nf}$ endowed with a symplectic structure $\omega$
is called a \emph{symplectic manifold} and denoted by $(\Vd{\Nf},\omega)$.}
%
%
The algebraic structure of a symplectic manifold $(\Vd{\Nf},\omega)$
can be characterized through the definition of a bracket:
Let $\exd \Fcal$ be the $1$-form given by the exterior derivative of
a given smooth function $\Fcal$. Then,
for all $\Fcal,\Gcal\in\Ld{\Nf}$,
\begin{equation}\label{eq:Jdef}
	\br{\Fcal}{\Gcal}{\Nf}:= \tensor[_{T^*\Vd{\Nf}\,}]{\langle}{}\exd \Fcal, \Jc{\Nf}\,\exd \Gcal\tensor[]{\rangle}{_{\,T\Vd{\Nf}}}
		= \omega(\Jc{\Nf}\,\exd \Fcal,\Jc{\Nf}\,\exd \Gcal),
\end{equation}
where $\tensor[_{T^*\Vd{\Nf}\,}]{\langle}{}\cdot, \cdot\tensor[]{\rangle}{_{\,T\Vd{\Nf}}}$ denotes the duality pairing between
the cotangent and the tangent bundle.
The \bl{function}
$\Jc{\Nf}:T^*\Vd{\Nf} \rightarrow T\Vd{\Nf}$ is a
contravariant $2$-tensor on the manifold $\Vd{\Nf}$,
commonly referred to as \emph{Poisson tensor}.
The space $\Ld{\Nf}$ of real-valued smooth functions over the manifold $(\Vd{\Nf},\br{\cdot}{\cdot}{\Nf})$,
together with the bracket $\br{\cdot}{\cdot}{\Nf}$, forms a Lie algebra \cite[Proposition 3.3.17]{AbMa78}.

\bl{To any function $\Hcal\in\Ld{\Nf}$, the symplectic form $\omega$ allows to associate a vector field $\Xcal_{\Hcal}\in T\Vd{\Nf}$, called \emph{Hamiltonian vector field}, via the relation}
\begin{equation}\label{eq:cont}
	d\Hcal=\contr_{\Xcal_{\Hcal}}\omega,
\end{equation}
where $\contr$ denotes the contraction operator.
Since $\omega$ is nondegenerate, $\Xcal_{\Hcal}\in T\Vd{\Nf}$ is unique.
\bl{Any} vector field $\Xcal_{\Hcal}$ on a manifold $\Vd{\Nf}$ determines a phase
flow, namely
a one-parameter group of diffeomorphisms
$\Phi^t_{\Xcal_{\Hcal}}:\Vd{\Nf}\rightarrow\Vd{\Nf}$ satisfying
$d_t\Phi^t_{\Xcal_{\Hcal}}(u)=\Xcal_{\Hcal}(\Phi^t_{\Xcal_{\Hcal}}(u))$
for all $t\in\Tcal$ and $u\in\Vd{\Nf}$, with $\Phi^0_{\Xcal_{\Hcal}}(u)=u$.
\bl{The flow of a Hamiltonian vector field
satisfies $(\Phi^t_{\Xcal_\Ham})^*\omega=\omega$, for each $t\in\Tcal$, that is
$\Phi^t_{\Xcal_\Ham}$ is a symplectic diffeomorphism (symplectomorphism) on its domain.}
%
\begin{definition}[Symplectic map]\label{def:SymplecticMap}
Let $(\Vd{\Nf},\br{\cdot}{\cdot}{\Nf})$
and
$(\Vd{\Nr},\br{\cdot}{\cdot}{\Nr})$ be symplectic manifolds of finite dimension
$\Nf$ and $\Nr$ respectively, with $\Nrh\leq \Nfh$.
A smooth map $\Psi:(\Vd{\Nf},\br{\cdot}{\cdot}{\Nf})\rightarrow(\Vd{\Nr},\br{\cdot}{\cdot}{\Nr})$
is called \emph{symplectic} if it satisfies
\begin{equation*}
	\Psi^*\br{\Fcal}{\Gcal}{\Nr}=\br{\Psi^*\Fcal}{\Psi^*\Gcal}{\Nf},
	\qquad \forall\,\Fcal, \Gcal \in\Ld{\Nr}.
\end{equation*}
\end{definition}
In addition to possessing a symplectic phase flow, Hamiltonian dynamics
is characterized by
the existence of differential invariants, and symmetry-related conservation laws.
\begin{definition}[Invariants of motion]\label{def:inv}
A function $\Ical\in\Ld{\Nf}$ is an \emph{invariant of motion} of 
the dynamical system \eqref{eq:cont}, 
if $\br{\Ical}{\HamN}{\Nf}(u)=0$ for all $u\in \Vd{\Nf}$.
Consequently, $\Ical$ is constant along the orbits of $\Xcal_{\HamN}$.
\end{definition}
The Hamiltonian, if time-independent, is an invariant of motion.
A particular subset of the invariants of motion of a dynamical system is
given by the \emph{Casimir invariants}, smooth functions $\Ccal$ on $\Vd{\Nf}$ that
$\br{\cdot}{\cdot}{\Nf}$-commute with every other functions, i.e.
$\br{\Ccal}{\Fcal}{\Nf}=0$ for all $\Fcal\in\Ld{\Nf}$.
Since Casimir invariants are associated with
the center of the Lie algebra $(\Ld{\Nf},\br{\cdot}{\cdot}{\Nf})$,
symplectic manifolds only possess trivial Casimir invariants.

Resorting to a coordinate system, the canonical structure on a symplectic manifold
can be characterized
by canonical charts whose existence is postulated in
\cite[Proposition 3.3.21]{AbMa78}.

\begin{definition}\label{def:CCoor}
Let $(\Vd{\Nf},\br{\cdot}{\cdot}{\Nf})$ be a symplectic manifold and $(U,\psi)$ a
cotangent coordinate chart
$\psi(u) = (q^1(u),\ldots,q^{\Nfh}(u),p_1(u),\ldots,p_{\Nfh}(u))$, for all $u\in U$.
Then $(U,\psi)$ is a \emph{symplectic canonical chart}
if and only if
$\{q^i,q^j\}_{\Nf}=\{p_i,p_j\}_{\Nf}=0$, and $\{q^i,p_j\}_{\Nf}=\delta_{i,j}$ on $U$ for all $i,j=1,\ldots,\Nfh$.
\end{definition}
In the local canonical coordinates introduced in Definition~\ref{def:CCoor}, the vector bundle map $\Jc{\Nf}$,
defined in \eqref{eq:Jdef},
takes the canonical symplectic form
\begin{equation*}
\J{\Nf} :=
	\begin{pmatrix}
		 0 & \Id \\
		-\Id & 0 \\
	\end{pmatrix}: T^*\Vd{\Nfh}\times T^*\Vd{\Nfh}\longrightarrow T\Vd{\Nf},
\end{equation*}
where $\Id$ and $0$ denote the identity and zero map, respectively.
%
\bl{Symplectic canonical charts on a symplectic vector space allow to identify a
\emph{K\"alher structure}, namely a compatible combination of a scalar product and symplectic form, as follows.}
\bl{On a symplectic vector space $(\Vd{\Nf},\omega)$,
the operator $\J{\Nf}^\top$ is an \emph{almost complex structure}, that is a linear map 
on $\Vd{\Nf}$ such that $\J{\Nf}^\top\circ\J{\Nf}^\top= -\Id$.
Furthermore, $\J{\Nf}^\top$
is \emph{compatible} with the symplectic structure $\omega$, namely,
for any $u,v\in\Vd{\Nf}$, $u\neq 0$, it holds
$$\omega(\J{\Nf}^\top u,\J{\Nf}^\top v)=\omega(u,v),\qquad\mbox{ and }\qquad \omega(u,\J{\Nf}^\top u)>0.$$
A symplectic form $\omega$ on a vector space $\Vd{\Nf}$ together with a compatible positive almost complex structure $\J{\Nf}^\top$ determines an inner product on $\Vd{\Nf}$, given by
\begin{equation}\label{eq:inprod}
	(u,v):=\omega(u,\J{\Nf}^\top v),\quad \forall\,u,v\in\Vd{\Nf}.
\end{equation}
A symplectic basis
on $(\Vd{\Nf},\omega)$
is an orthonormal basis for the compatible inner product \eqref{eq:inprod},
and we refer to it as \emph{orthosymplectic}.}
%
%
A subspace $\Ud$ of a symplectic vector space $(\Vd{\Nf},\omega)$
is called \emph{Lagrangian} if
it coincides with its symplectic complement in $\Vd{\Nf}$,
\bl{namely the set of $u\in\Vd{\Nf}$ such that $\omega(u,v)=0$ for all $v\in \bl{\Ud}$.}
As a consequence of the fact that any basis of a Lagrangian subspace
of a symplectic vector space
can be extended to a symplectic basis,
every symplectic vector space admits an orthosymplectic basis,
\emph{cf.} for example \cite[Section 1.2]{CdS01}.

With the definitions introduced hitherto,
we can recast the dynamical system \eqref{eq:dynIntro}
on a symplectic vector space $(\Vd{\Nf},\omega)$
as a Hamiltonian initial value problem.
For each $\prm\in\Sprm$,
and for $u_0(\prm)\in\Vd{\Nf}$, find $u(\cdot,\prm)\in C^1(\Tcal,\Vd{\Nf})$ such that
\begin{equation}\label{eq:dynN}
	\left\{
	\begin{array}{ll}
		\partial_t u(t,\prm) = \J{\Nf}\nabla_u \HamN(u(t,\prm);\prm), & \quad \quad\mbox{for }\;t\in\Tcal,\\
		u(t_0,\prm) = u_0(\prm),&
	\end{array}\right.	
\end{equation}
where  $\HamN(\cdot,\prm)\in\Ld{\Nf}$ is the Hamiltonian function,
and $\nabla_u$ denotes the gradient with respect to the variable $u$.
The well-posedness of \eqref{eq:dynN} is guaranteed by assuming that,
for any fixed $\prm\in\Sprm$, the operator 
$\Xcal_{\HamN}:\Vd{\Nf}\times\Sprm\rightarrow \r{}$ defined as
$\Xcal_{\HamN}(u,\prm):=\J{\Nf}\nabla_u \HamN(u;\prm)$
is Lipschitz continuous in $u$ uniformly in $t\in\Tcal$ in a suitable norm.

\section{Orthosymplectic matrices}\label{sec:OSmatrix}
In order to construct surrogate models
preserving the physical and geometric properties of the original Hamiltonian dynamics
we build approximation spaces of reduced dimension endowed
with the same geometric structure of the full model.
To this aim, the reduced space is constructed as the span of
suitable symplectic and orthonormal time-dependent bases,
so that the reduced space
inherits the \bl{geometric} structure of the original dynamical system.
In this Section we describe the properties of linear symplectic maps
between finite dimensional symplectic vector spaces.

Analogously to \cite[p. 168]{AbMa78}, we can easily extend the characterization of
symplectic linear maps to the case of vector spaces of different dimension as in the following result.
\begin{lemma}
Let $(\Vd{\Nf},\omega)$ and $(\Vd{\Nr},\omega)$ be symplectic vector spaces of
finite dimension $\Nf$ and $\Nr$, respectively, with $\Nfh \geq \Nrh$.
A linear \bl{map} $M_+:(\Vd{\Nf},\omega)\rightarrow(\Vd{\Nr},\omega)$ is symplectic, in the sense of Definition \ref{def:SymplecticMap},
if and only if \bl{the corresponding matrix representation $M_+\in\R{\Nr}{\Nf}$ satisfies}
$M_+\J{\Nf} M_+^\top = \J{\Nr}$.
\end{lemma}
We define
\emph{symplectic right inverse} of the symplectic matrix $M_+\in\R{\Nr}{\Nf}$ the matrix
$M=\J{\Nf} M_+^\top\J{\Nr}^\top\in\R{\Nf}{\Nr}$.
\bl{It can be easily verified that $M_+M=\Idm_{\Nr}$, and
that $M:(\Vd{\Nr},\omega)\rightarrow(\Vd{\Nf},\omega)$ is the adjoint operator
with respect to the symplectic form $\omega$, i.e. $\omega(M_+ u,y)=\omega(u,My)$
for any $u\in\Vd{\Nf}$, and $y\in\Vd{\Nr}$.
Furthermore, the symplectic condition
$M_+\J{\Nf} M_+^\top = \J{\Nr}$ is equivalent to $M^\top\J{\Nf} M = \J{\Nr}$.}
Owing to this equivalence,
with a small abuse of notation, we will say that $M\in\R{\Nf}{\Nr}$ is symplectic if \bl{it belongs to the space}
$$\Sp(\Nr,\r{\Nf}) := \{L\in\R{\Nf}{\Nr}:\;L^\top\J{\Nf} L = \J{\Nr}\}.$$

\begin{definition}\label{def:U}
A matrix $M\in\R{\Nf}{\Nr}$ is called \emph{orthosymplectic} if \bl{it belongs to the space}
$$\Un(\Nr,\r{\Nf}):=\St(\Nr,\r{\Nf})\cap \Sp(\Nr,\r{\Nf}),$$
where $\St(\Nr,\r{\Nf}) := \{M\in\R{\Nf}{\Nr}:\;M^\top M = \Idm_{\Nr}\}$ is the Stiefel manifold.
\end{definition}%
Orthosymplectic rectangular matrices can be characterized as follows.
\begin{lemma}\label{lem:MM+}
Let $M_+\in\R{\Nr}{\Nf}$ be symplectic and let $M\in\R{\Nf}{\Nr}$ be its symplectic inverse.
Then, $M_+ M_+^\top=\Idm_{\Nr}$ if and only if $M = M_+^\top$.
\end{lemma}
\begin{proof}
Let $M=[A\,|\,B]$ with $A,B\in\R{\Nf}{\Nrh}$.
The (semi-)orthogonality and symplecticity of $M_+$
give $A^\top A=B^\top B=\Idm_{\Nrh}$
and $A^\top \J{\Nf} B=\Idm_{\Nrh}$. These conditions imply that
the column vectors of $A$ and $\J{\Nf} B$ have unit norm and are pairwise parallel,
hence $A=\J{\Nf}B$. Therefore, $M=[A\,|\,\J{\Nf}^\top A]$
with $A^\top A=\Idm_{\Nrh}$ and $A^\top\J{\Nf} A=0_{\Nrh}$.
The definition of symplectic inverse yields
$M_+^\top = \J{\Nf}^\top M\J{\Nr}
= \J{\Nf}^\top [A\,|\,\J{\Nf}^\top A]\J{\Nr}
=[\J{\Nf}^\top A\,|\, {-A}]\J{\Nr}
=[A\,|\, \J{\Nf}^\top A] = M$.

Conversely, the symplecticity of $M_+$ implies
$M_+ M_+^\top = 
M_+ \J{\Nf} M_+^\top\J{\Nr}^\top = \Idm_{\Nr}$.
\end{proof}
In order to design numerical methods for evolution problems on
the manifold $\Un(\Nr,\r{\Nf})$ of orthosymplectic rectangular matrices,
we will need to
characterize its tangent space. 
To this aim we introduce the
vector space $\fr{so}(\Nr)$ of skew-symmetric $\Nr\times \Nr$ real matrices
$\fr{so}(\Nr) := \{M\in\R{\Nr}{\Nr}:\;M^\top+M = \Zrm_{\Nr}\}$,
and the vector space $\fr{sp}(\Nr)$ of Hamiltonian $\Nr\times \Nr$ real matrices, namely
$\fr{sp}(\Nr) := \{M\in\R{\Nr}{\Nr}:\;M\J{\Nr}+\J{\Nr}M^\top = \Zrm_{\Nr}\}$.
Throughout, if not otherwise specified, we will denote with $G_{\Nr}:=\Un(\Nr)$
the Lie group of orthosymplectic $\Nr\times\Nr$ matrices and
with $\fr{g}_{\Nr}$ the corresponding Lie algebra $\fr{g}_{\Nr}:=\fr{so}(\Nr)\cap \fr{sp}(\Nr)$,
with bracket given by the matrix commutator
\bl{$\ad{M}(L)=[M,L]:=ML-LM$, for any $M,L\in\fr{g}_{\Nr}$.}

\section{Orthosymplectic dynamical reduced basis method}\label{sec:SDLR}
Assume we want to solve the parameterized Hamiltonian problem \eqref{eq:dynN}
at $p\in\mathbb{N}$ samples of the parameter
$\{\prm_j\}_{j=1}^p=:\Gamma_h\subset\r{pd}$.
To \bl{simplify} the notation we take $d=1$, namely we assume that the parameter $\prm$
is a scalar quantity, for vector-valued $\prm$ the derivation
henceforth applies \bl{\textit{mutatis mutandis}}.
Then, the Hamiltonian system \eqref{eq:dynN} can be recast as a set of ordinary differential equations
in a $\Nf\times\Np$ matrix unknown. 
Let $\prm_h\in\r{\Np}$ denote the vector of sampled parameters,
the evolution problem reads:
For $\Rcal_0(\prm_h):=\big[u_0(\prm_1)|\ldots|u_0(\prm_{\Np})\big]\in\R{\Nf}{\Np}$,
find $\Rcal\in C^1(\Tcal,\R{\Nf}{\Np})$ such that
\begin{equation}\label{eq:dynNR}
	\left\{
	\begin{array}{ll}
		\dot{\Rcal}(t) = \Xcal_{\HamN}(\Rcal(t),\prm_h),&\quad\quad\mbox{for }\; t\in\Tcal,\\
		\Rcal(t_0) = \Rcal_0(\prm_h). &
	\end{array}\right.	
\end{equation}
Let $\Nrh\ll\Nfh$, to characterize the reduced solution manifold
we consider an approximation of the solution of \eqref{eq:dynNR} of the form
\begin{equation}\label{eq:Rrb}
	\Rcal(t)\approx R(t) 
		= \sum_{i=1}^{\Nr} \Ubf_i(t) \Zbf_i(t,\prm_h) = U(t)Z(t)^\top,
\end{equation}
where $U=\big[\Ubf_1|\ldots|\Ubf_{\Nr}\big]\in\R{\Nf}{\Nr}$, and
$Z\in\R{\Np}{\Nr}$
is such that
$Z_{j,i}(t)=\Zbf_i(t,\prm_j)$ for $i=1,\ldots,\Nr$, and $j=1,\ldots,\Np$.
Since we aim at a structure-preserving model order reduction of \eqref{eq:dynNR},
we impose that the basis $U(t)$ is orthosymplectic at all $t\in\Tcal$,
in analogy with the symplectic reduction techniques employing globally defined reduced spaces.
Here, since $U$ is changing in time, this means that
we constrain its evolution to the manifold $\Un(\Nr,\r{\Nf})$ from Definition \ref{def:U}.
With this \bl{in} mind, the reduced solution is sought in
the reduced space
defined as
\begin{equation}\label{eq:SRM}
	\Mcal^{\spl}_{\Nr} := \{R\in\R{\Nf}{\Np}:\; R = UZ^\top\;\mbox{with}\;
		U\in\Mcal,\, Z\in V^{\Np\times\Nr} \},
\end{equation}
where
\begin{equation}\label{eq:ManUZ}
\begin{aligned}
	\Mcal& :=\Un(\Nr,\r{\Nf})=\{U\in\R{\Nf}{\Nr}:\;U^\top U=\Idm_{\Nr},\; U^\top \J{\Nf} U = \J{\Nr}\},\\
	V^{\Np\times\Nr} & :=\{Z\in\R{\Np}{\Nr}:\;\rank{Z^\top Z + \J{\Nr}^\top Z^\top Z\J{\Nr}} = \Nr\}.
\end{aligned}
\end{equation}
Note that \eqref{eq:SRM} is a smooth manifold of dimension $2(\Nfh+\Np)\Nrh-2\Nrh^2$,
as follows from the characterization of the tangent space given in Proposition \ref{prop:TM}.
The characterization of the reduced manifold \eqref{eq:SRM} is analogous to \cite[Definition 6.2]{MN17}.
Let $C\in\R{\Nr}{\Nr}$ denote the correlation matrix $C:=Z^\top Z$.
The full-rank condition in \eqref{eq:ManUZ},
\begin{equation}\label{eq:fullRank}
	\rank{C + \J{\Nr}^\top C\J{\Nr}} = \Nr,
\end{equation}
guarantees that, for $Z$ fixed, if $UZ^\top=WZ^\top$ with $U,W\in\Mcal$, then $U=W$.
If the full-rank condition \eqref{eq:fullRank} is satisfied, then
the number $\Np$ of samples of the parameter $\prm\in\Sprm$ satisfies $\Np\geq \Nrh$.
This means that, for a fixed $\Np$,
a too large reduced basis
might lead to a violation of
the full rank condition, which would entail
a rank-deficient evolution problem for the coefficient matrix $Z\in\R{\Np}{\Nr}$.
This is related to the problem of overapproximation in dynamical low-rank techniques, see \cite[Section 5.3]{KL07}.
\bl{Observe also that if $\Np\geq\Nr$ and $\rank{Z}=\Nr$ then the full rank condition \eqref{eq:fullRank} is always satisfied.
In general, the elements of $\Mcal^{\spl}_{\Nr}$ might not have full rank $\Nr$:
for any $R\in\Mcal^{\spl}_{\Nr}$ it holds
$\rank{Z}\leq \rank{R}\leq \min\{\Nr,\Np\}$.}

The decomposition $UZ^\top$ of matrices in $\Mcal^{\spl}_{\Nr}$
is not unique:
the map $\phi: (U,Z)\in \Mcal\times V^{\Np\times\Nr}\mapsto R=UZ^\top\in \Mcal^{\spl}_{\Nr}$
is surjective but not injective. In particular, $(\Mcal\times V^{\Np\times\Nr},\Mcal^{\spl}_{\Nr},\phi, \Un(\Nr))$ is a fiber bundle with fibers given by
the group of unitary matrices $\Un(\Nr)$, and $\Mcal^{\spl}_{\Nr}$ is isomorphic to
$(\Mcal/\Un(\Nr))\times V^{\Np\times\Nr}$.
%
Indeed, let $U_1\in\Mcal$ and $Z_1\in V^{\Np\times\Nr}$,
then, for any arbitrary $A\in\Un(\Nr)$, it holds $U_2:=U_1 A\in \Mcal$,
$Z_2:=Z_1 A\in V^{\Np\times\Nr}$,
and $U_1 Z_1^\top = U_2 Z_2^\top$.

In dynamically orthogonal approximations \cite{SL09}
a characterization of the reduced solution is obtained by fixing a gauge constraint
in the tangent space of the reduced solution manifold.
%
For the manifold $\Mcal^{\spl}_{\Nr}$ the tangent space at $R\in\Mcal^{\spl}_{\Nr}$
is defined as the set of $X\in\R{\Nf}{\Np}$ such that there exists a differentiable path
$\gamma:(-\varepsilon,\varepsilon)\subset\Tcal\rightarrow\R{\Nf}{\Np}$
with
$\gamma(0)=R$, $\dot{\gamma}(0)=X$.
The tangent vector at
$U(t)Z^\top(t)\in\Mcal^{\spl}_{\Nr}$ is of the form
$X = \dU Z^\top+U\dZ^\top$,
where $\dU$ and $\dZ$ denote the time derivatives of $U(t)$ and $Z(t)$, respectively.
Taking the derivative of the orthogonality constraint on $U$ yields $\dU^\top U+U^\top\dU = 0$.
Analogously, the symplecticity constraint gives $\dU^\top\J{\Nf}U+U^\top\J{\Nf}\dU=0$
which is equivalent to $\dU^\top U\J{\Nr}+\J{\Nr}U^\top\dU=0$
owing to the fact that $U\in\Sp(\Nr,\r{\Nf})$.
Therefore, the tangent space of $\Mcal^{\spl}_{\Nr}$ at $UZ^\top$ is defined as
\begin{equation}\label{eq:TM}
\begin{aligned}
	\T{UZ^\top}{\Mcal^{\spl}_{\Nr}} = \{X\in\R{\Nf}{\Np}:\; & X =X_U Z^\top + UX_Z^\top\;\,
	\mbox{with}\;\,
		X_Z\in\R{p}{\Nr},\\
		&\;X_U\in\R{\Nf}{\Nr},\,X_U^\top U\in\fr{g}_{\Nr}\}.
\end{aligned}
\end{equation}
However, this parameterization is not unique. Indeed, let $S\in\fr{g}_{\Nr}$ be arbitrary:
if $X_U^\top U\in\fr{g}_{\Nr}$ then the matrix $(X_U+US)^\top U$ belongs to $\fr{g}_{\Nr}$,
and
the pairs $(X_U,X_Z)$ and $(X_U+US,X_Z+ZS)$ identify the same tangent vector $X:=X_UZ^\top+UX_Z^\top$.
We fix the parameterization of the tangent space 
as follows.
%
\begin{proposition}\label{prop:TM}
The tangent space of $\Mcal^{\spl}_{\Nr}$ at $UZ^\top$ defined in \eqref{eq:TM}
is uniquely parameterized by the space $H_{(U,Z)} := \TsM{U} \times \R{\Np}{\Nr}$, where
\begin{equation}\label{eq:HU}
	\TsM{U} := \{X_U\in \R{\Nf}{\Nr}:\;X_U^\top U=0,\,X_U\J{\Nr}=\J{\Nf}X_U\}.
\end{equation}
This means that the map
\begin{equation*}
\begin{array}{lcll}
	\Psi: & H_{(U,Z)} & \longrightarrow & \T{UZ^\top}{\Mcal^{\spl}_{\Nr}}\\
		  & (X_U,X_Z) & \longmapsto &  X_U Z^\top + UX_Z^\top,
\end{array}
\end{equation*}
is a bijection.
\end{proposition}
\begin{proof}
We first observe that, if $(X_U,X_Z)\in H_{(U,Z)}$ then
$X_U^\top U\in\fr{g}_{\Nr}$ is trivially satisfied,
and hence 
$X_U Z^\top + UX_Z^\top\in \T{UZ^\top}{\Mcal^{\spl}_{\Nr}}$.

To show that the map $\Psi$ is injective, we take $X=0\in \T{UZ^\top}{\Mcal^{\spl}_{\Nr}}$.
By the definition of the tangent space \eqref{eq:TM}, the zero vector admits the representation
$0=X_U Z^\top + UX_Z^\top$ with $U^\top X_U=0$.
This implies $0=U^\top(X_U Z^\top + UX_Z^\top)=X_Z^\top$.
Hence, $X_U Z^\top=0$ and
\begin{equation*}
	0 = X_U Z^\top Z+\J{\Nf}X_UZ^\top Z\J{\Nr}^\top
	  = X_U Z^\top Z+\J{\Nf}^\top X_U \J{\Nr}\J{\Nr}Z^\top Z\J{\Nr}^\top
	  = X_U (Z^\top Z+ \J{\Nr} Z^\top Z \J{\Nr}^\top),
\end{equation*}
which implies $X_U=0$ in view of the full-rank condition \eqref{eq:fullRank}.

For the surjectivity of $\Psi$ we show that
\begin{equation*}
	\forall\, X\in \T{UZ^\top}{\Mcal^{\spl}_{\Nr}}\qquad
		\exists\, (X_U,X_Z)\in H_{(U,Z)}\quad\mbox{such that}\quad X = X_U Z^\top + UX_Z^\top.
\end{equation*}
Any $X\in \T{UZ^\top}{\Mcal^{\spl}_{\Nr}}$ can be written as $X = \dU Z^\top + U\dZ^\top$
where $\dZ\in\R{\Np}{\Nr}$ and $\dU\in\R{\Nf}{\Nr}$ satisfies
$\dU^\top U\in\fr{g}_{\Nr}$. 
Hence, the tangent vector $X$ can be recast as
\begin{equation*}
	 X = \dU Z^\top + U\dZ^\top = U(\dZ^\top + U^\top \dU Z^\top) + \big((\Idm_{\Nf}-UU^\top)\dU\big) Z^\top.
\end{equation*}
We need to show that the pair $(X_U,X_Z)$, defined as
$X_U:=(\Idm_{\Nf}-UU^\top)\dU$ and $X_Z:=\dZ + Z \dU^\top U$,
belongs to the space $H_{(U,Z)}$.
From the orthogonality of $U$ it easily follows that
\begin{equation*}
	U^\top X_U = U^\top(\Idm_{\Nf}-UU^\top)\dU = U^\top \dU-U^\top\dU = 0.
\end{equation*}
To prove that $X_U=\J{\Nf}^\top X_U\J{\Nr}$, we
introduce the matrix $S:=Z^\top Z+\J{\Nr}Z^\top Z\J{\Nr}^\top\in\R{\Nr}{\Nr}$
for which it holds $S\J{\Nr}=\J{\Nr}S$.
We then show the equivalent condition
$X_U S\J{\Nr}^\top = \J{\Nf}^\top X_US$.
First,
we add to $X_U$ the zero term $(\Idm_{\Nf}-UU^\top)U(\dZ^\top Z + \J{\Nr}\dZ^\top Z\J{\Nr}^\top)$,
and use the symplectic constraint on $U$ and its temporal derivative to get
\begin{equation*}
\begin{aligned}
	X_U & = (\Idm_{\Nf}-UU^\top)\dU = (\Idm_{\Nf}-UU^\top)\dU SS^{-1}\\
		& = 	(\Idm_{\Nf}-UU^\top)\big(U(\dZ^\top Z + \J{\Nr}\dZ^\top Z\J{\Nr}^\top)
			+ \dU (Z^\top Z+\J{\Nr}Z^\top Z\J{\Nr}^\top)\big)S^{-1}\\
		& = (\Idm_{\Nf}-UU^\top)(XZ+\J{\Nf} XZ\J{\Nr}^\top)S^{-1}.
\end{aligned}
\end{equation*}
Then, using the commutativity of the symplectic unit $\J{\Nf}$ and
the projection onto the orthogonal complement to the space spanned by $U$,
i.e. $(\Idm_{\Nf}-UU^\top)\J{\Nf} = \J{\Nf}(\Idm_{\Nf}-UU^\top)$, results in
\begin{equation*}
\begin{aligned}
	X_US\J{\Nr}^\top & = (\Idm_{\Nf}-UU^\top)(XZ+\J{\Nf} XZ\J{\Nr}^\top)\J{\Nr}^\top\\
		& = \J{\Nf}^\top(\Idm_{\Nf}-UU^\top)\J{\Nf}(XZ\J{\Nr}^\top+\J{\Nf}^\top XZ)
		= \J{\Nf}^\top X_US.
\end{aligned}
\end{equation*}
\end{proof}
\bl{\begin{remark}
Proposition \ref{prop:TM} provides a \emph{connection} on the
fiber bundle $(\Mcal\times V^{\Np\times\Nr},\Mcal^{\spl}_{\Nr},\phi,\Un(\Nr))$
via the smooth splitting
$\T{UZ^\top}{\Mcal^{\spl}_{\Nr}}= V_{(U,Z)} \oplus H_{(U,Z)}$, for any 
$UZ^\top\in\Mcal^{\spl}_{\Nr}$.
The factor $V_{(U,Z)}$, the \textit{vertical space}, is the subspace of $\T{UZ^\top}{\Mcal^{\spl}_{\Nr}}$ that consists of all vectors tangent to the fiber of $UZ^\top$, while the space $H_{(U,Z)} := \TsM{U} \times \R{\Np}{\Nr}$,
with $\TsM{U}$ defined in \eqref{eq:HU},
is a \textit{horizontal space}.
This decomposition into the subset of directions tangent to the fiber and its complementary space provides a unique parameterization of the tangent space.
We refer the reader to e.g. \cite[Chapter 2]{KN63} and \cite{EAS99}, for further details on the topic.
%
\end{remark}}
\bl{Owing to} Proposition \ref{prop:TM},
the tangent space of $\Mcal^{\spl}_{\Nr}$
can be characterized as
\begin{equation*}
\begin{aligned}
	\T{UZ^\top}{\Mcal^{\spl}_{\Nr}} = \{X\in\R{\Nf}{\Np}:\;
		& X = X_U Z^\top + UX_Z^\top\;\, \mbox{with}\;\,
		 X_Z\in\R{p}{\Nr},\\
		&\, X_U\in\R{\Nf}{\Nr},
		 \,X_U^\top U=0,\,X_U\J{\Nr}=\J{\Nf}X_U\},
\end{aligned}
\end{equation*}
Henceforth, we consider $\Mcal$ endowed with the metric
induced by the ambient space $\C{\Nf}{\Nr}$, namely the Frobenius
inner product $\langle A,B\rangle:=\tr(A^\hop B)$,
where $A^\hop$ denotes the conjugate transpose of the complex matrix $A$,
and we will denote with $\norm{\cdot}$ the Frobenius norm.
\bl{Note that, on simple Lie algebras, the Frobenius inner product is a multiple of the Killing form.}

\subsection{Dynamical low-rank symplectic variational principle}
For any fixed $\prm\in\Sprm$, the vector field $\Xcal_{\HamN}$
in \eqref{eq:dynN} at time $t$ belongs to $\T{u(t)}{\Vd{\Nf}}$.
Taking the cue from
dynamical low-rank approximations \cite{KL07},
we derive a dynamical system on the reduced space $\Mcal^{\spl}_{\Nr}$ via
projection of the velocity field $\Xcal_{\HamN}$ of the full dynamical system \eqref{eq:dynNR}
onto the tangent space of $\Mcal^{\spl}_{\Nr}$ at the current state.
The reduced dynamical system is therefore optimal in the sense that
the resulting vector field is the best dynamic approximation of $\Xcal_{\HamN}$, in the Frobenius norm,
at every point on the manifold $\Vd{\Nf}$.
To preserve the geometric structure of the full dynamics we
construct a projection which is symplectic for each value of the parameter $\prm_j\in\Sprmh$,
with $1\leq j\leq \Np$.
To this aim, let us introduce on the symplectic vector space $(\Vd{\Nf},\omega)$
the family of skew-symmetric bilinear forms
$\omega_j:\R{\Nf}{\Np}\times\R{\Nf}{\Np}\rightarrow\r{}$
defined as
\begin{equation}\label{eq:omegaj}
	\omega_j(a,b):=\omega(a_j,b_j),\qquad 1\leq j\leq \Np,
\end{equation}
where $a_j\in\r{\Nf}$ denotes the $j$-th column
of the matrix $a\in\R{\Nf}{\Np}$, and similarly for $b_j\in\r{\Nf}$.
\begin{proposition}\label{prop:SymplProj}
Let $\T{R}{\Mcal^{\spl}_{\Nr}}$ be 
the tangent space of the symplectic reduced manifold $\Mcal^{\spl}_{\Nr}$, defined in \eqref{eq:SRM},
at a given $R:=UZ^\top\in\Mcal^{\spl}_{\Nr}$.
Let $S:=Z^\top Z+\J{\Nr}Z^\top Z\J{\Nr}\in\R{\Nr}{\Nr}$.
Then, the map
\begin{equation*}
\begin{array}{lcll}
	\Pi_{\T{R}{\Mcal^{\spl}_{\Nr}}}: & \R{\Nf}{\Np} & \longrightarrow & \T{R}{\Mcal^{\spl}_{\Nr}}\\
		  & w & \longmapsto & (\Idm_{\Nf}-UU^\top)(wZ + \J{\Nf}wZ\J{\Nr}^\top)S^{-1}Z^\top+ UU^\top w,
\end{array}
\end{equation*}
is a symplectic projection, in the sense that
\begin{equation*}
	\sum_{j=1}^p \omega_j\big(w-\Pi_{\T{R}{\Mcal^{\spl}_{\Nr}}}w,y\big) = 0,\qquad \forall\, y\in\T{R}{\Mcal^{\spl}_{\Nr}},
\end{equation*}
where $\omega_j$ is defined in \eqref{eq:omegaj}.
\end{proposition}
\begin{proof}
Let $X_U(w):=(\Idm_{\Nf}-UU^\top)(wZ + \J{\Nf}wZ\J{\Nr}^\top)(Z^\top Z+\J{\Nr}Z^\top Z\J{\Nr}^\top)^{-1}$
and $X_Z(w) = w^\top U$.
Using a reasoning analogous to the one in the proof of Proposition \ref{prop:TM}, it can be shown that
$(X_U,X_Z)\in H_{(U,Z)}$.
Moreover,
by means of the identification $T\T{R}{\Mcal^{\spl}_{\Nr}}\cong\T{R}{\Mcal^{\spl}_{\Nr}}$,
we prove that $\Pi:=\Pi_{\T{R}{\Mcal^{\spl}_{\Nr}}}$
is a projection.
It can be easily verified that $X_Z(\Pi w) = (\Pi w)^\top U= X_Z(w)$.
Furthermore, let
$F_w:=wZ + \J{\Nf}wZ\J{\Nr}^\top\in\R{\Nf}{\Nr}$, then
\begin{equation*}
\begin{aligned}
	X_U(\Pi w)&  = (\Idm_{\Nf}-UU^\top)
		\big((\Idm_{\Nf}-UU^\top)F_w S^{-1}Z^\top Z
		 +
		\J{\Nf}(\Idm_{\Nf}-UU^\top)F_w S^{-1}Z^\top Z\J{\Nr}^\top\big)S^{-1}\\
		& = X_U(w)Z^\top ZS^{-1} + \J{\Nf}X_U(w)Z^\top Z\J{\Nr}^\top S^{-1}.
\end{aligned}
\end{equation*}
Since $X_U(w)\J{\Nr} = \J{\Nf}X_U(w)$, it follows that
$X_U(\Pi w) 
= X_U(w)$.

Assume we have fixed a parameter $\prm_j\in\Sprm$ so that $p=1$.
Let $v:=w_j\in\r{\Nf}$ be the $j$-th column of the matrix $w\in\R{\Nf}{p}$ and, hence, $\Pi v\in\r{\Nf}$.
We want to show that $\omega(v-\Pi v, y) = 0$ for all $y\in\T{R}{\Mcal^{\spl}_{\Nr}}$.
By the characterization of the tangent space from Proposition \ref{prop:TM},
any $y\in\T{R}{\Mcal^{\spl}_{\Nr}}$ is of the form
$y=Y_U Z^\top + U Y_Z^\top$ where $Y_Z\in\R{1}{\Nr}$ and $Y_U\in \TsM{U}$.
Therefore,
\begin{equation*}
	\omega(v-\Pi v, y) = \omega(v-\Pi v,Y_UZ^\top)
		 + \omega(v, U Y_Z^\top) - \omega(X_U Z^\top + U X_Z^\top, U Y_Z^\top),
\end{equation*}
where $X_U=X_U(v)$ and $X_Z=X_Z(v)$, but henceforth we omit the dependence on $v$.
Using the definition of $X_Z$ and the symplecticity of the basis $U$ the last term becomes
\begin{equation*}
\begin{aligned}
	\omega(U X_Z^\top, U Y_Z^\top) & = \omega(UU^\top v,U Y_Z^\top)
		= \omega(v,\J{\Nf}^\top U \J{\Nr}U^\top U Y_Z^\top)\\
		& = \omega(v,\J{\Nf}^\top U\J{\Nr} Y_Z^\top)
		= \omega(v, U Y_Z^\top).
\end{aligned}
\end{equation*}
Moreover, it can be easily checked that $\omega(X_U Z^\top, U Y_Z^\top)=0$ by definition of $X_U$
and by the orthosymplecticity of $U$.
Hence, the only non-trivial terms are $\omega(v-\Pi v, y) = \omega(v ,Y_UZ^\top) - \omega(\Pi v ,Y_UZ^\top)$.
Any $Y_U\in \TsM{U}$ can be written as $Y_U = \frac12 (Y_U + \J{\Nf}^\top Y_U\J{\Nr})$; thereby
\begin{equation*}
\begin{aligned}
	\omega(v-\Pi v, 2y)  
		= \, \omega(v,Y_UZ^\top + \J{\Nf}^\top Y_U\J{\Nr}Z^\top)
		  -\omega(X_U Z^\top + U X_Z^\top,Y_UZ^\top + \J{\Nf}^\top Y_U\J{\Nr}Z^\top)
		 =: T_1 - T_2.
\end{aligned}
\end{equation*}
We need to prove that $T_1$ and $T_2$ coincide.
Let $M_i\in\r{\Nf}$ denote the $i$-th column vector of a given matrix $M\in\R{\Nf}{\Nr}$.
The properties of the symplectic canonical form $\omega$ yield
\begin{equation*}
\begin{aligned}
	T_1 & = \omega\bigg(v, \sum_{i=1}^{\Nr} (Y_U)_i Z_i\bigg) +
			\omega\bigg(\J{\Nf}v, \sum_{i=1}^{\Nr} (Y_U)_i (\J{\Nr}Z^\top)_i\bigg)\\
		& = \sum_{i=1}^{\Nr} \omega\big(v,(Y_U)_i\big)Z_i + \sum_{i=1}^{\Nr} \omega\big(\J{\Nf}v,(Y_U)_i\big)(\J{\Nr}Z^\top)_i
		 = \sum_{i=1}^{\Nr} \omega\big(v Z_i + \J{\Nf}v(Z\J{\Nr}^\top)_i,(Y_U)_i\big).
\end{aligned}
\end{equation*}
To deal with the term $T_2$ first observe that $\omega(U X_Z^\top,Y_U Z^\top) = 0$ since $Y_U^\top U = 0$.
Moreover, using once more the fact that $Y_U\in \TsM{U}$ results in
\begin{equation*}
\begin{aligned}
	T_2 & = \omega\big(X_U Z^\top, Y_U Z^\top\big) + \omega\big(X_U\J{\Nr} Z^\top, Y_U\J{\Nr} Z^\top\big)\\
		&  = \sum_{i,j=1}^{\Nr} \omega\big((X_U)_j Z_j, (Y_U)_i Z_i\big) + \omega\big((X_U)_j (\J{\Nr}Z^\top)_j, (Y_U)_i (\J{\Nr}Z^\top)_i\big)\\
		&  = \sum_{i,j=1}^{\Nr} \omega\big((X_U)_j, (Y_U)_i\big)\big(Z_j Z_i +  (\J{\Nr}Z^\top)_j (Z\J{\Nr}^\top)_i\big).
\end{aligned}
\end{equation*}
The result follows by definition of $X_U(v)$.
\end{proof}
\bl{\begin{remark}
Owing to the inner product structure \eqref{eq:inprod}, the projection operator from Proposition~\ref{prop:SymplProj} is orthogonal in the Frobenius norm since
\begin{equation*}
\sum_{j=1}^p \omega_j\big(w-\Pi_{\T{R}{\Mcal^{\spl}_{\Nr}}}w,y\big) = 
	\langle w-\Pi_{\T{R}{\Mcal^{\spl}_{\Nr}}}w,\J{\Nf}y\rangle = 0,\qquad \forall\, y\in\T{R}{\Mcal^{\spl}_{\Nr}}.
\end{equation*}
This means that the projection gives the best low-rank approximation of the velocity vector,
and hence the reduced dynamics is associated with the flow field ensuing from the best approximation
in the tangent space to the reduced manifold.
\end{remark}}
To compute the initial condition of the reduced problem, we perform the
complex SVD of $\Rcal_0(\prm_h)\in\R{\Nf}{\Np}$ truncated at the $\Nrh$-th mode. Then the
initial value $U_0\in\Mcal$ is obtained from the resulting unitary matrix of left singular vectors of 
$\Rcal_0(\prm_h)$ by exploiting the isomorphism between $\Mcal$ and $\St(\Nrh,\c{\Nfh})$,
\emph{cf.} Lemma \ref{lem:isoM}.
The expansion coefficients matrix is initialized as $Z_0 = \Rcal_0(\prm_h)^\top U_0$.
Therefore, the dynamical system for the approximate reduced solution \eqref{eq:Rrb}
reads: Find $R\in C^1(\Tcal,\Mcal^{\spl}_{\Nr})$ such that
\begin{equation}\label{eq:dynn}
	\left\{
	\begin{array}{ll}
		\dot{R}(t) = \Pi_{\T{R}{\Mcal^{\spl}_{\Nr}}}\Xcal_{\HamN}(R(t),\prm_h),&\quad\quad\mbox{for }\; t\in\Tcal,\\
		R(t_0) = U_0Z_0^\top. &
	\end{array}\right.	
\end{equation}
For any $1\leq j\leq \Np$ and $t\in\Tcal$,
let $Z_j(t)\in\R{1}{\Nr}$ be the $j$-th row of the matrix $Z(t)\in V^{\Np\times\Nr}$, 
and let $Y(t):=[Y_1|\ldots|Y_p]\in\R{\Nf}{\Np}$
where $Y_j := \nabla_{UZ_j^\top}\HamN(UZ_j^\top,\prm_j)\in\R{\Nf}{1}$,
and $\nabla_{UZ_j^\top}$ denotes the gradient with respect to $UZ_j^\top$.
Using the decomposition $R=UZ^\top$ in \eqref{eq:SRM}, we can now derive from \eqref{eq:dynn} evolution equations for $U$ and $Z$:
Given $\Rcal_0(\prm_h)\in\R{\Nf}{\Np}$,
find $(U,Z)\in C^1(\Tcal,\Mcal)\times C^1(\Tcal,V^{\Np\times\Nr})$ such that
\begin{equation}\label{eq:eqU}
\left\{
\begin{array}{ll}
	\dot{Z}_j(t) = \J{\Nr}\nabla_{Z_j} \HamN(UZ_j^\top,\prm_j), & \quad t\in\Tcal,\; 1\leq j\leq p,\\
	\dot{U}(t) = (\Idm_{\Nf}-UU^\top)(\J{\Nf}YZ - YZ\J{\Nr}^\top)S^{-1}, &\quad t\in\Tcal,\\
	U(t_0)Z(t_0)^\top = U_0Z_0^\top. &
\end{array}
\right.
\end{equation}
The reduced problem \eqref{eq:eqU} is analogous to the system derived in \cite[Proposition 6.9]{MN17}.
The evolution equations for the coefficients $Z$
form a system of
$p$ equations in $\Nr$ unknowns and correspond to the Galerkin projection
onto the space spanned by the columns of $U$, as obtained with a standard reduced basis method.
Here, however, the projection is changing over time
as the reduced basis $U$ is evolving.
For $U$ fixed, the flow map characterizing the evolution of each $Z_j$, for $1\leq j\leq p$, is a symplectomorphism
(\emph{cf.} Definition \ref{def:SymplecticMap}),
i.e. the dynamics is canonically Hamiltonian.
The evolution problem satisfied by the basis $U$ is
a matrix equation in $\Nf\times\Nr$ unknowns
on the manifold of orthosymplectic rectangular matrices introduced in Definition \ref{def:U},
as shown in the following result.
\begin{proposition}\label{prop:OSflow}
If $U(t_0)\in\Mcal$ then $U(t)\in\R{\Nf}{\Nr}$ solution of
\eqref{eq:eqU} satisfies $U(t)\in \Mcal$ for all $t\in\Tcal$.
\end{proposition}
\begin{proof}
We first show that, for any matrix $W(t)\in\R{\Nf}{\Nr}$, if $W(t_0)\in\Mcal$
and $\dot{W}\in \TsM{W}$, with $\TsM{W}$ defined in \eqref{eq:HU}, then $W(t)\in\Mcal$ for any $t>t_0$.
The condition $\dot{W}^\top W=0$ implies
$d_t(W^\top(t) W(t)) = \dot{W}^\top W + W^\top\dot{W} = 0$,
hence $W^\top(t) W(t)=W^\top(t_0) W(t_0)=\Idm_{\Nr}$ by the assumption on the initial condition.
Moreover, the condition $\dot{W}=\J{\Nf}^\top \dot{W}\J{\Nr}$ together with the dynamical orthogonality $\dot{W}^\top W=0$ results in
$d_t(W^\top(t)\J{\Nf} W(t)) = \dot{W}^\top\J{\Nf} W + W^\top\J{\Nf}\dot{W}
	= \J{\Nr}^\top\dot{W}^\top W + W^\top\dot{W}\J{\Nr}^\top = 0$.
Hence, the symplectic constraint on the initial condition yields
$W^\top(t)\J{\Nf} W(t)=W^\top(t_0)\J{\Nf} W(t_0)=\J{\Nr}$.

Owing to the reasoning above, we only need to verify that the solution of \eqref{eq:eqU} satisfies
$\dot{U}\in \TsM{U}$.
The dynamical orthogonal condition $\dot{U}^\top U=0$ is trivially satisfied.
Moreover,
since $S\J{\Nr}=\J{\Nr}S$,
the constraint $\dot{U}=\J{\Nf}^\top \dot{U}\J{\Nr}$ is satisfied if
$\dot{U}S\J{\Nr}^\top= \J{\Nf}^\top\dot{U}S$.
One can easily show that
$A:= \J{\Nf}YZ- YZ\J{\Nr}^\top = \J{\Nf}A\J{\Nr}^\top$.
Therefore, $	\dot{U}S\J{\Nr}^\top = (\Idm_{\Nf}-UU^\top)A\J{\Nr}^\top
= \J{\Nf}^\top (\Idm_{\Nf}-UU^\top)\J{\Nf}A\J{\Nr}^\top = \J{\Nf}^\top\dot{U} S$.
\end{proof}
\begin{remark}
Observe that the dynamical reduced basis technique proposed in the previous Section can be
extended to more general
Hamiltonian systems endowed with
a degenerate constant Poisson structure.
The idea is to proceed as in \cite[Section 3]{HP18} by splitting the dynamics
into the evolution on a symplectic submanifold of the phase space and the trivial evolution of the Casimir invariants.
The symplectic dynamical model order reduction
developed in Section \ref{sec:SDLR} can then be performed on the symplectic component of the dynamics.
\end{remark}

\subsection{Conservation properties of the reduced dynamics}
The velocity field of the reduced flow \eqref{eq:dynn}
is the symplectic projection of the full model velocity onto the tangent space
of the reduced manifold.
For any fixed parameter $\prm_j\in\Sprmh$, 
let $\HamN_j:=\HamN(\cdot,\prm_j)$.
In view of Proposition~\ref{prop:SymplProj}, the reduced solution
$R\in C^1(\Tcal,\Mcal^{\spl}_{\Nr})$
satisfies the symplectic variational principle
\begin{equation*}
	\sum_{j=1}^p \omega_j\big(\dot{R}-\J{\Nf}\nabla\HamN_j(R),y\big) = 0,\qquad \forall\, y\in\T{R}{\Mcal^{\spl}_{\Nr}}.
\end{equation*}
This implies that
the Hamiltonian $\HamN$ is a conserved quantity of the continuous reduced problem
\eqref{eq:eqU}. Indeed,
\begin{equation*}
	\sum_{j=1}^p\dfrac{d}{dt}\HamN_j(R(t))  = \sum_{j=1}^p(\nabla_{R_j}\HamN_j(R),\dot{R}_j) 
		= \sum_{j=1}^p\omega(\J{\Nf}\nabla_{R_j}\HamN_j(R),\dot{R}_j)
		= \sum_{j=1}^p\omega_j(\dot{R},\dot{R}) = 0.
\end{equation*}
Therefore, if $\Rcal_0(\prm_h)\in\spn{U_0}$ then the Hamiltonian is preserved,
\begin{equation*}
	\sum_{j=1}^p\big(\HamN_j(\Rcal(t))-\HamN_j(R(t))\big) =
	\sum_{j=1}^p\big(\HamN_j(\Rcal_0)-\HamN_j(R(t_0))\big)
	 =	\sum_{j=1}^p\big(\HamN_j(\Rcal_0)-\HamN_j(U_0U_0^\top \Rcal_0)\big).
\end{equation*}
To deal with the other invariants of motion,
let us assume for simplicity that $p=1$.
Since the linear map $\r{\Nf}\rightarrow \spn{U(t)}$ associated with the reduced basis 
at any time $t\in\Tcal$ cannot be symplectic, the invariants of motion of the
full and reduced model cannot be in one-to-one correspondence.
Nevertheless, a result analogous to \cite[Lemma 3.9]{HP18} holds.
\begin{lemma}
Let $\pi_{+,t}^\ast$ be the pullback of the linear map associated with the reduced basis $U^\top(t)$
at time $t\in\Tcal$.
Assume that $\HamN\in\Im{\pi_{+,t}^\ast}$ for any $t\in\Tcal$.
Then, $\Ical(t)\in C^\infty(\r{\Nr})$ is an invariant of $\Phi^t_{X_{\pi_{+,t}^\ast \HamN}}$
if and only if $(\pi_{+,t}^\ast \Ical)(t)\in C^\infty(\r{\Nf})$ is an invariant of $\Phi^t_{X_{\HamN}}$ 
in $\Im{\pi_{+,t}^\ast}$.
\end{lemma}

\subsection{Convergence estimates with respect to the best low-rank approximation}
In order to derive error estimates for the reduced solution
of problem \eqref{eq:dynn}, we extend to our setting
the error analysis of \cite[Section 5]{FL18}
which shows that the error committed by the
dynamical approximation with respect to the best low-rank approximation
is bounded by the projection error of the full model solution onto
the reduced manifold of low-rank matrices.
To this aim, we resort to the isomorphism between
the reduced symplectic manifold $\Mcal^{\spl}_{\Nr}$ defined in \eqref{eq:SRM}
and the manifold $\Mcal_{\Nrh}$ of rank-$\Nrh$ complex matrices, already established in
\cite[Lemma 6.1]{MN17}.
Then, we derive the dynamical orthogonal approximation of the resulting problem
in the complex setting and prove that it is isomorphic to the solution of the reduced Hamiltonian system
\eqref{eq:dynn}. The differentiability properties of orthogonal projections onto smooth
embedded manifolds and the trivial extension to complex matrices of the curvature bounds
in \cite{FL18} allows to derive an error estimate.

Let $\mathfrak{L}(\Omega)$ denote the set of functions with values in the vector space $\Omega$,
and let $\mathfrak{F}:\mathfrak{L}(\R{\Nf}{\Np})\rightarrow\mathfrak{L}(\C{\Nfh}{\Np})$ be the isomorphism
\begin{equation}\label{eq:Frak}
	R(\cdot)=
\begin{pmatrix}
R_q(\cdot)\\[-0.4em]
R_p(\cdot)
\end{pmatrix}	
		\;\longmapsto\;
			\mathfrak{F}(R)(\cdot)=R_q(\cdot)+iR_p(\cdot).
\end{equation}
Then, problem \eqref{eq:dynNR} can be recast in the complex setting as:
For $\Rcal_0(\prm_h)\in\R{\Nf}{\Np}$,
find $\Ccal\in C^1(\Tcal,\C{\Nfh}{\Np})$ such that
\begin{equation}\label{eq:dynNRcomplex}
	\left\{
	\begin{array}{ll}
		\dot{\Ccal}(t) = \mathfrak{F}(\Xcal_{\HamN})(\Ccal(t),\prm_h)
		=:\widehat{\Xcal}_{\HamN}(\Ccal(t),\prm_h),&\quad\quad\mbox{for }\; t\in\Tcal,\\
		\Ccal(t_0) = \mathfrak{F}(\Rcal_0)(\prm_h). &
	\end{array}\right.	
\end{equation}
Similarly to dynamically orthogonal approximations 
we consider the manifold of rank-$\Nrh$ complex matrices
$\Mcal_{\Nrh}:=\{C\in\C{\Nfh}{\Np}:\;\rank{C}=\Nrh\}$.
Any $C\in\Mcal_{\Nrh}$
can be decomposed, up to unitary $\Nrh\times\Nrh$ transformations, as
$C=WY^\top$ where $W\in\St(\Nrh,\c{\Nfh})=\{M\in\C{\Nfh}{\Nrh}:\;M^\hop M=\Idm_{\Nrh}\}$,
and $Y\in \Vcal^{\Np\times\Nrh}:=\{M\in\C{\Np}{\Nrh}:\;\rank{M}=\Nrh \}$.
Analogously to \cite[Lemma 6.1]{MN17} one can establish the following result.
\begin{lemma}\label{lem:isoM}
	The manifolds $\Mcal_{\Nrh}$ and $\Mcal^{\spl}_{\Nr}$ 
	are isomorphic via the map
	\begin{equation}\label{eq:isoM}
		(U,Z)\in\Mcal \times V^{\Np\times\Nr}
			\;\longmapsto\;
			(\mathfrak{F}(A),\mathfrak{F}(Z^\top)^\top)\in \St(\Nrh,\c{\Nfh})\times\Vcal^{\Np\times\Nrh},
	\end{equation}
	where $\mathfrak{F}$ is defined in \eqref{eq:Frak} and $A\in\R{\Nf}{\Nrh}$ is such that $U=[A\,|\,\J{\Nf}^\top A]$
	in view of Lemma \ref{lem:MM+}.
\end{lemma}
For $C(t_0)\in \Mcal_{\Nrh}$
associated with $R(t_0)\in\Mcal^{\spl}_{\Nr}$ via the map \eqref{eq:isoM},
we can therefore derive the DO dynamical system:
find $C\in C^1(\Tcal,\Mcal_{\Nrh})$ such that
\begin{equation}\label{eq:dynncomplex}
		\dot{C}(t) = \Pi_{\T{C}{\Mcal_{\Nrh}}}\widehat{\Xcal}_{\HamN}(C(t),\prm_h),\quad\quad\mbox{for }\; t\in\Tcal,
\end{equation}
where $\Pi_{\T{C}{\Mcal_{\Nrh}}}$ is the projection onto the tangent space
of $\Mcal_{\Nrh}$ at $C=WY^\top$, defined as
\begin{equation*}
\begin{aligned}
	\T{C}{\Mcal_{\Nrh}} = \{X\in\C{\Nfh}{\Np}:\; & X=X_W Y^\top +W X_Y^\top\;\,
	\mbox{with}\;\,
		X_Y\in\C{p}{\Nrh},\\
		&\;X_W\in\C{\Nfh}{\Nrh},\,X_W^\hop W+W^\hop X_W=0\}.
\end{aligned}
\end{equation*}
The so-called dynamically orthogonal condition $X_W^\hop W=0$,
allows to uniquely parameterize the tangent space $\T{C}{\Mcal_{\Nrh}}$ by imposing
that the complex reduced basis evolves orthogonally to itself.

Let $M^\ast$ indicate the complex conjugate of a given matrix $M$. The projection
onto the tangent space of $\Mcal_{\Nrh}$ can be characterized as in the following result.
\begin{lemma}
At every $C=WY^\top\in\Mcal_{\Nrh}$, the map
\begin{equation}\label{eq:TmapC}
\begin{array}{lcll}
	\Pi_{\T{C}{\Mcal_{\Nrh}}}: & \C{\Nfh}{\Np} & \longrightarrow & \T{C}{\Mcal_{\Nrh}}\\
		  & w & \longmapsto & (\Idm_{\Nfh}-WW^\hop)w\,Y^\ast(Y^\top Y^\ast)^{-1}Y^\top+ WW^\hop w,
\end{array}
\end{equation}
is the $\norm{\cdot}$-orthogonal projection onto the tangent space of $\Mcal_{\Nrh}$ at $C$.
\end{lemma}
\begin{proof}
The result can be derived similarly to the proof of \cite[Proposition 7]{FL18}
by minimizing the convex functional $\mathfrak{J}(X_W,X_Y):=\frac12\norm{w-X_W Y^\top -W X_Y^\top}^2$
under the constraint $X_W^\hop W=0$.
\end{proof}
Using the expression \eqref{eq:TmapC} for the projection onto the tangent space of $\Mcal_{\Nrh}$,
we can derive from \eqref{eq:dynncomplex} evolution equations for the terms $W$ and $Y$:
Given $C_0=\Pi_{\Mcal_{\Nrh}}\Ccal(t_0)\in\C{\Nfh}{\Np}$ orthogonal projection onto $\Mcal_{\Nrh}$,
find $(W,Y)\in C^1(\Tcal,\St(\Nrh,\c{\Nfh}))\times C^1(\Tcal,\Vcal^{\Np\times\Nrh})$ such that
\begin{equation}\label{eq:eqUcomplex}
\left\{
\begin{array}{ll}
	\dot{Y}^\ast(t) = \widehat{\Xcal}_{\HamN}^\hop(WY^\top,\prm_h)W, & \quad t\in\Tcal,\\
	\dot{W}^\ast(t) = (\Idm_{\Nfh}-W^\ast W^\top)\widehat{\Xcal}_{\HamN}^\ast(WY^\top,\prm_h)Y(Y^\hop Y)^{-1}, &\quad t\in\Tcal.
\end{array}
\right.
\end{equation}
\begin{proposition}\label{prop:EquivRC}
Under the assumption of well-posedness,
problem \eqref{eq:dynn} is equivalent to problem \eqref{eq:dynncomplex}.
\end{proposition}
\begin{proof}
The proof easily follows from algebraic manipulations of the field equations \eqref{eq:eqU}
and \eqref{eq:eqUcomplex} and from the definition of the isomorphism
\eqref{eq:isoM}.
\end{proof}

In view of Proposition \ref{prop:EquivRC}, we can revert to the error estimate established in \cite{FL18}.
\begin{theorem}[{\cite[Theorem 32]{FL18}}]
Let $\Ccal\in C^1(\Tcal,\C{\Nfh}{\Np})$ denote the exact solution of \eqref{eq:dynNRcomplex} and
let $C\in C^1(\Tcal,\Mcal_{\Nrh})$ be the solution of \eqref{eq:dynncomplex} at time $t\in\Tcal$.
Assume that no crossing of the singular values of $\Ccal$ occurs, namely
\begin{equation*}
	\sigma_n(\Ccal(t)) > \sigma_{n+1}(\Ccal(t)),\qquad \forall\, t\in\Tcal.
\end{equation*}
Let $\Pi_{\Mcal_{\Nrh}}$ be the $\norm{\cdot}$-orthogonal projection onto $\Mcal_{\Nrh}$.
Then, at any time $t\in\Tcal$, the error between the approximate solution $C(t)$ and the best rank-$\Nrh$ approximation of
$\Ccal(t)$ can be bounded as
\begin{equation*}
	\norm{C(t)-\Pi_{\Mcal_{\Nrh}}\Ccal(t)}
		\leq \int_{\Tcal}\nu
		\norm{\Ccal(s)-\Pi_{\Mcal_{\Nrh}}\Ccal(s)} e^{\mu(t-s)}\,ds,
\end{equation*}
where $\mu\in\r{}$ and $\nu\in\r{}$ are defined as
\begin{equation*}
	\mu:=L_{\Xcal} + 2\, \sup_{t\in\Tcal} \dfrac{\norm{\Xcal_{\HamN}(\Ccal(t),\eta_h)}}{\sigma_{n}(\Ccal(t))}\,,
	\qquad\;
	\nu:=L_{\Xcal}+\dfrac{\norm{\Xcal_{\HamN}(\Ccal(s),\eta_h)}}{\sigma_{n}(\Ccal(s))-\sigma_{n+1}(\Ccal(s))},
\end{equation*}
and $L_{\Xcal}\in\r{}$ denotes the Lipschitz continuity constant of $\Xcal_{\HamN}$.
\end{theorem}
The remainder of this work pertains to numerical methods for the temporal discretization of the reduced dynamics
\eqref{eq:eqU}.
Since we consider splitting techniques, see e.g. \cite[Section II.5]{HaLuWa06},
the evolution problems for the expansion coefficients
and for the reduced basis are examined separately.
The coefficients $Z(t)\in V^{\Np\times\Nr}$ of the expansion \eqref{eq:Rrb}
satisfy a Hamiltonian dynamical system \eqref{eq:eqU} in the reduced symplectic manifold of dimension $\Nr$
spanned by the evolving orthosymplectic basis $U(t)\in\Mcal$.
The numerical approximation of the evolution equation for $Z(t)$ can, thus, be performed using symplectic integrators,
\emph{cf.} \cite[Section VI]{HaLuWa06}.
%
Observe that the use of standard splitting techniques might require
the approximate reduced solution, at a given time step, to be projected
into the space spanned by the updated basis.
This might cause an error in the conservation of the invariants due to the projection step,
that, however, can be controlled under sufficiently small time steps.
In principle, exact conservation can be guaranteed if the evolution of the reduced basis
evolves smoothly at the interface of temporal interval (or temporal subintervals associated with the splitting),
or, in other words, if the splitting is synchronous and the two systems
are concurrently advanced in time.
We postpone to future work the investigation and the numerical study of
splitting methods that exactly preserve the Hamiltonian.

\section{Numerical methods for the evolution of the reduced basis}\label{sec:EvolROM}
Contrary to global projection-based model order reduction,
dynamical reduced basis methods eschew
the standard online-offline paradigm.
The construction and evolution of the local reduced basis \eqref{eq:eqU}
does not require queries of the high-fidelity model
so that the method does not incur a computationally expensive offline phase.
However, the evolution of the reduced basis
entails the solution of a matrix equation
in which one dimension equals the size of the full model.
Numerical methods
for the solution of \eqref{eq:eqU} will have
arithmetic complexity $\min\{C_{\Rcal},C_{\Fcal}\}$
where $C_{\Fcal}$ is the computational cost required
to evaluate the velocity field of \eqref{eq:eqU},
and $C_{\Rcal}$ denotes the cost associated with all other operations. 
Assume that the cost to evaluate the Hamiltonian at the reduced solution
has order $O(\alpha(\Nfh))$.
Then, a standard algorithm for the evaluation of the right hand side of \eqref{eq:eqU}
will have arithmetic complexity $C_{\Fcal} = O(\alpha(\Nfh))+O(\Nfh \Nrh^2) +O(\Nfh p\,\Nrh) + O(\Nrh^3)$,
where the last two terms are associated with the computation
of $YZ$,
and the inversion of $C+\J{\Nr}^\top C\J{\Nr}$, respectively.
%
This Section focuses on the development of structure-preserving numerical methods
for the solution of \eqref{eq:eqU}
such that $C_{\Rcal}$ is at most linear in $\Nfh$.
The efficient treatment of the nonlinear terms is out of the scope of the present study and
will be the subject of future investigations on
structure-preserving hyper-reduction techniques.

To \bl{simplify} the notation, we recast \eqref{eq:eqU}
as: For $Q\in\Mcal$, find $U\in C^1(\Tcal,\R{\Nf}{\Nr})$ 
such that
\begin{equation}\label{eq:dynNbasis}
	\left\{
	\begin{array}{ll}
		\dot{U}(t) = \Fcal(U(t)),&\quad\quad\mbox{for }\; t\in\Tcal,\\
		U(t_0) = Q, &
	\end{array}\right.	
\end{equation}
where, for any fixed $t\in\Tcal$,
\begin{equation}\label{eq:Fcal}
	\Fcal(U):=(\Idm_{\Nf}-UU^\top)
		(\J{\Nf}YZ-YZ\J{\Nr}^\top)(Z^\top Z+\J{\Nr}^\top Z^\top Z\J{\Nr})^{-1}.
\end{equation}
\bl{Observe that $\Fcal:U\in\Mcal\mapsto \Fcal(U)\in \TsM{U}\subset\TM{U}$,
where $H_U$ is defined as in \eqref{eq:HU}, and
$\TM{U}  = \{V\in\R{\Nf}{\Nr}:\;U^\top V\in \fr{g}_{\Nr}\}.$}
In a temporal splitting perspective, we assume that the matrix $Z(t)\in V^{\Np\times\Nr}$ is given at each time instant $t\in\Tcal$.
Owing to Proposition \ref{prop:OSflow}, if $Q\in\Mcal$, then $U(t)\in\Mcal$ for all $t\in\Tcal$.
Then, the goal is to develop an efficient numerical scheme such that the discretization of \eqref{eq:dynNbasis}
yields an approximate flow map with trajectories belonging to $\Mcal$.

We propose two intrinsic numerical methods for the solution of
the differential equation \eqref{eq:dynNbasis}
within the class of numerical methods based on local charts on manifolds \cite[Section IV.5]{HaLuWa06}.
The analyticity and the favorable computational properties of the Cayley transform, \emph{cf.} Proposition \ref{prop:Nn2}
and \cite{Iserles01},
makes it our choice as coordinate map on the orthosymplectic matrix manifold.

\subsection{Cayley transform as coordinate map}\label{sec:cay}
Orthosymplectic square matrices form a subgroup $\Un(\Nf)$ of a quadratic Lie group.
We can, therefore, use the Cayley transform to induce a local parameterization of the
Lie group $\Un(\Nf)$ near the identity, with the corresponding Lie algebra as parameter space.
The following results extend to orthosymplectic matrices the properties of the Cayley
transform presented in e.g. \cite[Section IV.8.3]{HaLuWa06}.
\begin{lemma}
Let $\Gcal_{\Nf}$ be the group of orthosymplectic square matrices and
let $\fr{g}_{\Nf}$ be the corresponding Lie algebra.
Let $\cay: \fr{g}_{\Nf}\rightarrow \R{\Nf}{\Nf}$ be the Cayley transform
defined as
\begin{equation}\label{eq:cay}
	\cay(\Omega) = \left(\Idm-\dfrac{\Omega}{2}\right)^{-1}\left(\Idm+\dfrac{\Omega}{2}\right),
	\qquad\forall\, \Omega\in\fr{g}_{\Nf}.
\end{equation}
Then,
\begin{itemize}
	\item[(i)] $\cay$ maps the Lie algebra $\fr{g}_{\Nf}$ into the Lie group $\Gcal_{\Nf}$.
	\item[(ii)] $\cay$ is a diffeomorphism in a neighborhood of the zero matrix $0\in\fr{g}_{\Nf}$.
	The differential of $\cay$ at $\Omega\in\fr{g}_{\Nf}$ is the map $\dcay_\Omega:\T{\Omega}{\fr{g}_{\Nf}}\cong\fr{g}_{\Nf}\rightarrow\T{\cay(\Omega)}{\Gcal_{\Nf}}$,
	\begin{equation*}
		\dcay_\Omega(A) = \left(\Idm-\dfrac{\Omega}{2}\right)^{-1}A\left(\Idm+\dfrac{\Omega}{2}\right)^{-1},
		\qquad \forall\, A\in\fr{g}_{\Nf},
	\end{equation*}
	and its inverse is
	\begin{equation}\label{eq:dcayinv}
		\dcay_\Omega^{-1}(A) = \left(\Idm-\dfrac{\Omega}{2}\right)A\left(\Idm+\dfrac{\Omega}{2}\right),
		\qquad \forall\, A\in\T{\cay(\Omega)}{\Gcal_{\Nf}}.
	\end{equation}
	\item[(iii)] \cite[Theorem 3]{DLP98} Let $\sigma(A)$ denote the spectrum of $A\in\R{\Nf}{\Nf}$.
	If $\Omega\in C^1(\r{},\fr{g}_\Nf)$ then $A:=\cay(\Omega)\in C^1(\r{},\Gcal_{\Nf})$.
	Conversely, if $A\in C^1(\r{},\Gcal_{\Nf})$ and $-1\notin\bigcup_{t\in\r{}}\sigma(A(t))$
	then there exists a unique $\Omega\in C^1(\r{},\fr{g}_\Nf)$ such that
	\bl{$\Omega=\cay^{-1}(A)=2(A-\Idm_{\Nf})(A+\Idm_{\Nf})^{-1}$}.
\end{itemize}
\end{lemma}
\begin{proof}
Let $\Omega\in\fr{g}_{\Nf}$ and let $\overline{\Omega}:=\Omega/2$. Since $\Omega$ is skew-symmetric then $I-\overline{\Omega}$ is invertible.\\
(i) The Cayley transform defined in \eqref{eq:cay} can be recast as
\begin{equation}\label{eq:cay1}
\begin{aligned}
	\cay(\Omega) & = -(\Idm-\overline{\Omega})^{-1}(-2\Idm+(\Idm-\overline{\Omega})) = 2(\Idm-\overline{\Omega})^{-1}-\Idm\\
			& = -(-2\Idm+(\Idm-\overline{\Omega}))(\Idm-\overline{\Omega})^{-1} = (\Idm+\overline{\Omega})(\Idm-\overline{\Omega})^{-1}.
\end{aligned}
\end{equation}
Then, using \eqref{eq:cay1} and the skew-symmetry of $\Omega\in\fr{g}_{\Nf}$ results in
\begin{equation*}
\begin{aligned}
	\cay(\Omega)^\top \cay(\Omega) & 
						   = (\Idm-\overline{\Omega})^{-\top}(\Idm + \overline{\Omega}^\top \overline{\Omega})(\Idm-\overline{\Omega})^{-1}\\
			& = (\Idm-\overline{\Omega})^{-\top}(\Idm -\overline{\Omega}-\overline{\Omega}^\top + \overline{\Omega}^\top \overline{\Omega})(\Idm-\overline{\Omega})^{-1} = \Idm.
\end{aligned}
\end{equation*}
Moreover, 
$\cay(\Omega)\J{\Nf} = \J{\Nf}\cay(\Omega)$ since
\begin{equation*}
\begin{aligned}
	\cay(\Omega)\J{\Nf} & 
					 = (\Idm +\overline{\Omega})(-\J{\Nf}+\overline{\Omega}\J{\Nf})^{-1}
	 			     = (\Idm +\overline{\Omega})(-\J{\Nf}+\J{\Nf}\overline{\Omega})^{-1}\\
	 			   & = (\Idm +\overline{\Omega})\J{\Nf}(I-\overline{\Omega})^{-1}
					 = (\J{\Nf}-\J{\Nf}\overline{\Omega}^\top)(\Idm-\overline{\Omega})^{-1}\\
				   & = (\J{\Nf}+\J{\Nf}\overline{\Omega})(\Idm-\overline{\Omega})^{-1}
				     = \J{\Nf}\,\cay(\Omega).
\end{aligned}
\end{equation*}
(ii) The map $\cay$ \eqref{eq:cay} has non-zero derivative at $0\in\fr{g}_{\Nf}$.
Therefore, by the inverse function theorem, it is a diffeomorphism in a neighborhood of $0\in\fr{g}_{\Nf}$.
Standard rules of calculus yield the expression \eqref{eq:dcayinv}, \emph{cf.} \cite[Section IV.8.3, Lemma 8.8]{HaLuWa06}.
\end{proof}
The factor $1/2$ in the definition \eqref{eq:cay} of the Cayley transform is arbitrary and
has been introduced to guarantee that $\dcay_0 = \Idm_{\Nf}$, which
will be used in Section \ref{sec:TangMeth} for the construction of retraction maps.

To derive computationally efficient numerical schemes for the solution of the basis evolution equation \eqref{eq:dynNbasis}
we exploit the properties of analytic functions evaluated at the product of rectangular matrices.
\begin{proposition}\label{prop:Nn2}
\bl{Let $\Omega\in\fr{g}_\Nf$ and $Y\in\R{\Nf}{r}$.
If $\Omega$ has rank $\Nm\leq \Nf$, then
$\cay(\Omega)Y$ can be evaluated
with computational complexity of order $O(\Nfh r\Nm) + O(\Nm^2 r) + O(\Nm^3)$.}
\end{proposition}
\begin{proof}
\bl{Since $\Omega$ has rank $\Nm$ it admits the splitting $\Omega =\alpha\beta^\top$ for some $\alpha,\,\beta\in\R{\Nf}{\Nm}$.}
To evaluate the Cayley transform
in a computationally efficient way we exploit the properties of
analytic functions of low-rank matrices.
More in details, let $f(z):=z^{-1}(\cay(z)-1)$ for any $z\in\c{}$.
\bl{The function $f$ has a removable pole at $z = 0$. Its analytic extension reads,
\begin{equation*}
	f(z) = \sum_{m=0}^\infty 2^{-m} z^m.
\end{equation*}}
For any $m\in\mathbb{N}\setminus\{0\}$ it holds
$\Omega^m=(\alpha\beta^\top)^m = \alpha(\beta^\top\alpha)^{m-1}\beta^\top$. 
Hence,
\begin{equation*}
\begin{aligned}
	\cay(\Omega) 
			& = \Idm_{\Nf} + \sum_{m=1}^\infty 2^{1-m} \Omega^m
			= \Idm_{\Nf} + \sum_{m=1}^\infty 2^{1-m} \alpha (\beta^\top\alpha)^{m-1}\beta^\top\\
			& = \Idm_{\Nf} + \alpha f(\beta^\top\alpha)\beta^\top.
\end{aligned}
\end{equation*}
The cost to compute $A:=\beta^\top\alpha\in\R{\Nm}{\Nm}$ is $O(\Nfh\Nm^2)$.
Moreover,
\begin{equation*}
	\cay(\Omega) Y = (\Idm_{\Nfh} + \alpha f(\beta^\top\alpha)\beta^\top)Y
			 = Y + \alpha (\beta^\top\alpha)^{-1}\big(\cay(\beta^\top\alpha)-\Idm_{\Nm}\big)\beta^\top Y.
\end{equation*}
The evaluation of $f(A) = A^{-1}(\cay(A)-\Idm_{\Nm})\in\R{\Nm}{\Nm}$ requires $O(\Nm^3)$ operations.
Finally, the matrix multiplications $\alpha f(A)\beta^\top Y$ 
can be performed 
in \bl{$O(\Nfh r \Nm)+O(\Nm^2 r)$} operations.

The approach suggested hitherto is clearly not unique. The invertibility of the matrix $A$ is ensured under the condition
that the low-rank factors $\alpha$ and $\beta$ are full rank. Although a low-rank decomposition with full rank factors
is achievable \cite[Proposition 4]{CI00}, one could alternatively envision the use of 
Woodbury matrix identity \cite{Wood50} to compute the matrix inverse appearing in the
definition \eqref{eq:cay} of the Cayley transform.
This yields the formula
\begin{equation*}
	\cay(\Omega) Y = Y + \dfrac12\alpha \big(\cay(\beta^\top\alpha)+\Idm_{\Nm}\big)\beta^\top Y
	= Y - \alpha \bigg(\dfrac12 \beta^\top\alpha-\Idm_{\Nm}\bigg)^{-1}\beta^\top Y,
\end{equation*}
which can also be evaluated in \bl{$O(\Nfh r\Nm)+O(\Nm^2 r)+O(\Nm^3)$} operations.
\end{proof}


\subsection{Numerical integrators based on Lie groups \bl{acting on manifolds}}
In this Section we propose a numerical scheme for the solution
of \eqref{eq:dynNbasis} based on Lie group methods, \bl{\emph{cf.} \cite{IMNZ00}.
The idea is to consider $\Mcal$ as a manifold acted upon by the Lie group
$\Gcal_{\Nf}=\Un(\Nf)$ of square orthosymplectic matrices.
%
Then, since the local structure in a neighbourhood of any point of $\Gcal_{\Nf}$ can be described by the corresponding Lie algebra $\fr{g}_{\Nf}$, a local coordinate map is employed
to derive a differential equation on $\fr{g}_{\Nf}$.
Since Lie algebras are linear spaces,
using Runge--Kutta methods to solve the equation on $\fr{g}_{\Nf}$
allows to derive discrete trajectories that remain on the Lie algebra.
This approach falls within the class of numerical
integration schemes based on canonical coordinates of
the first kind, also known as Runge--Kutta Munthe-Kaas (RK-MK) methods
\cite{MK95,MKZ97,MK98,MK99}.}
\begin{proposition}\label{prop:Gsyst}
The evolution equation \eqref{eq:dynNbasis}
with arbitrary $\Fcal:\Mcal\rightarrow\TM{}$
is equivalent to the problem:
For $Q\in\Mcal$, find $U\in C^1(\Tcal,\Mcal)$ such that
\begin{equation}\label{eq:dynNbasisGr}
	\left\{
	\begin{array}{ll}
		\dot{U}(t) = \Lcal(U(t))\, U(t),&\quad\quad\mbox{for }\; t\in\Tcal,\\
		U(t_0) = Q, &
	\end{array}\right.
\end{equation}
\bl{with $\Lcal:\Mcal\rightarrow\fr{g}_\Nf$} defined as
\begin{equation}\label{eq:GS}
	\Lcal(U) := \dfrac12 \big(\Scal(U) + \J{\Nf}^\top \Scal(U)\J{\Nf}\big),
\end{equation}
where
$\Scal(U) := (\Idm_{\Nf} - U U^\top) \Fcal(U) U^\top -U \Fcal(U)^\top$.
\bl{Furthermore, if $\Fcal(U)\in \TsM{U}$, for any $U\in\Mcal$, then,}
\begin{equation}\label{eq:GSFcal}
	\Lcal(U) = \Fcal(U)U^\top - U\Fcal(U)^\top.
\end{equation}
\end{proposition}
\begin{proof}
\bl{Let us consider, at each time $t\in\Tcal$, an orthosymplectic extension
$Y(t)\in\R{\Nf}{\Nf}$ of $U(t)$
by the matrix $W(t)\in\R{\Nf}{2(\Nfh-\Nrh)}$,
such that $Y(t)=[U(t)\,|\,W(t)]\in \Un(\Nf)$.
Since $Y$ is orthosymplectic by construction, it holds
\begin{equation*}
\begin{array}{lll}
	0 & = \dfrac{d}{dt}(Y^\top Y) = \dot{Y}^\top Y+Y^\top \dot{Y},
		& \quad\Longrightarrow\quad \dot{Y} = -Y\dot{Y}^\top Y,\\[0.5em]
	0 & = \dfrac{d}{dt}(Y^\top\J{\Nf} Y) =\dot{Y}^\top\J{\Nf} Y+Y^\top\J{\Nf} \dot{Y},
		& \quad\Longrightarrow\quad \dot{Y} = -\J{\Nf}^\top Y\dot{Y}^\top\J{\Nf} Y.
\end{array}	
\end{equation*}
It follows that
$\dot{Y}(t) = \Acal(Y,\dot{Y}) Y(t)$, for all $t\in\Tcal$,
with
\begin{equation*}
	\Acal(Y,\dot{Y}) := -\dfrac12\big(Y\dot{Y}^\top + \J{\Nf}^\top Y\dot{Y}^\top\J{\Nf}\big)\in\fr{g}_{\Nf},
\end{equation*}
and $Y\dot{Y}^\top=U\dot{U}^\top+W\dot{W}^\top$.
Expressing $\Acal(Y,\dot{Y})$ explicitly in terms of $U$ and $W$,
and using the evolution equation satisfied by $U$, yields
\begin{equation}\label{eq:G}
	\Acal(Y,\dot{Y}) = -\dfrac12\big(U\Fcal(U)^\top +\J{\Nf}^\top U \Fcal(U)^\top\J{\Nf}
		+ W\dot{W}^\top + \J{\Nf}^\top W\dot{W}^\top\J{\Nf}\big).
\end{equation}
Moreover, since $\Acal$ is skew-symmetric, it holds
\begin{equation}\label{eq:W}
	\dot{W}W^\top  + W\dot{W}^\top =  - U \Fcal(U)^\top - \Fcal(U) U^\top.
\end{equation}
If $W\in\R{\Nf}{2(\Nfh-\Nrh)}$ is such that
$\dot{W}W^\top=-U\Fcal(U)^\top (\Idm_{\Nf}-UU^\top)$, then \eqref{eq:W}
it satisfied, owing to the fact that $\Fcal(U)\in \TM{U}$.
Substituting this expression in \eqref{eq:G} yields
expression \eqref{eq:GS} with $\Lcal(U) = \Acal(Y,\dot{Y})$.

Finally, if $\Fcal(U)$ belongs to $\TsM{U}$ then $U^\top\Fcal(U)=0$.
Substituting in \eqref{eq:GS} and using the fact that $\J{\Nf}^\top\Fcal(U) U^\top\J{\Nf}=\Fcal(U) U^\top$
yields \eqref{eq:GSFcal}.}
\end{proof}

\bl{Once we have recast \eqref{eq:dynNbasis} into the equivalent problem \eqref{eq:dynNbasisGr},
the idea is to derive an evolution equation on the Lie algebra $\fr{g}_{\Nf}$ via a coordinate map.
A coordinate map of the first kind is a smooth function
$\coo:\fr{g}_{\Nf}\rightarrow \Gcal_{\Nf}$ such that $\coo(0)=\Id\in\Gcal_{\Nf}$ and $\dcoo_0=\Id$, where
$\dcoo:\fr{g}_{\Nf}\times\fr{g}_{\Nf}\rightarrow\fr{g}_{\Nf}$ is the right trivialized tangent of $\coo$ defined as
\begin{equation}\label{eq:dcoo}
\dfrac{d}{dt} \coo(A(t))=\dcoo_{A(t)}(\dot{A}(t))\coo(A(t)),
\qquad \forall\, A:\mathbb{R}\rightarrow \fr{g}_{\Nf}.
\end{equation}
For sufficiently small $t\geq t_0$, the solution of \eqref{eq:dynNbasisGr}
is given by
$U(t) = \bl{\coo}(\Omega(t))U(t_0)$ where $\Omega(t)\in\fr{g}_{\Nf}$
satisfies
\begin{equation}\label{eq:Nalgebr}
	\left\{
	\begin{array}{ll}
		\dot{\Omega}(t) = \bl{\dcoo}_{\Omega(t)}^{-1}\big(\Lcal\big(U(t)\big)\big),&\quad\quad\mbox{for }\; t\in\Tcal,\\
		\Omega(t_0) = \bl{0}. &
	\end{array}\right.
\end{equation}

}

\bl{Problem \eqref{eq:Nalgebr} can be solved using traditional RK methods.}
Let $(b_i,a_{i,j})$, for $i=1,\ldots,s$
and $j=1,\ldots,\Ns$,
be the coefficients of the Butcher tableau describing an $\Ns$-stage explicit RK method.
Then, the numerical approximation of \eqref{eq:Nalgebr} in the interval $(t^m,t^{m+1}]$ is performed as in Algorithm~\ref{algo:RKMK}.

\begin{algorithm}
\caption{Explicit RK-MK scheme in $(t^m,t^{m+1}]$}\label{algo:RKMK}
\bl{\normalsize
\begin{algorithmic}[1]
 \Require{$U_m\in\Mcal$, $\{b_i\}_{i=1}^{\Ns}$, $\{a_{i,j}\}_{i,j=1}^{\Ns}$ }
 \State $\Omega_{m}^1=0$, $U_m^1=U_m$
 \For{$i=2,\ldots,\Ns$}
 \State $\mmi = \dt\sum\limits_{j=1}^{i-1} a_{i,j} \, \bl{\dcoo}_{\Omega_m^j}^{-1}\big(\Lcal(\umj)\big)$,\label{line:Mi}
 \State $\umi = \bl{\coo}(\mmi)\,U_m$,\label{line:Ui}
 \EndFor
 \State $\Omega_{m+1} = \dt\sum\limits_{i=1}^\Ns b_i\,\bl{\dcoo}_{\mmi}^{-1}\big(\Lcal(\umi)\big)$,
 \State \Return $U_{m+1} = \bl{\coo}(\Omega_{m+1})\,U_m\in\Mcal$
\end{algorithmic}}
\end{algorithm}
\bl{
As anticipated in Section~\ref{sec:cay}, we resort to the Cayley transform as coordinate map in Algorithm~\ref{algo:RKMK}.
The use of the Cayley transform in the solution of matrix differential equations on Lie groups was proposed in
\cite{DLP98,LP01,Iserles01}.
Analogously to \cite[Theorem 5]{DLP98}, it can be shown that
the invertibility of $\cay$ and $\dcay$ is guaranteed if
$U(t)\in\Mcal$ solution of \eqref{eq:dynNbasisGr} satisfies $-1\notin\bigcup_{t\in\Tcal}\sigma(U(t))$.}
Note that choosing a sufficiently small time step for the temporal integrator
can prevent the numerical solution from having an eigenvalue close to $-1$,
for some $t\in\Tcal$.
Alternatively, restarting procedures of the Algorithm~\ref{algo:RKMK} can be implemented
similarly to \cite[pp. 323-324]{DLP98}.

The computational cost of Algorithm~\ref{algo:RKMK} with $\coo=\cay$ is assessed in the following result.
\bl{\begin{proposition}\label{prop:RKcay}
Consider the evolution problem \eqref{eq:Nalgebr} on a fixed temporal interval $(t^m,t^{m+1}]\subset\Tcal$. Assume that the problem is solved with
Algorithm~\ref{algo:RKMK} where the coordinate map $\coo$ is
given by the Cayley transform $\cay$ defined in \eqref{eq:cay}.
Then, the computational complexity of the resulting scheme is of order
$O(\Nfh\Nrh^2\Ns^2)+O(\Nrh^3\Ns^4)+C_{\Fcal}$,
where $C_{\Fcal}$ is the
complexity of the algorithm to compute $\Fcal(U)$ in \eqref{eq:Fcal} at any given $U\in\Mcal$.
\end{proposition}
\begin{proof}
We need to assess the computational cost of
two operations in Algorithm~\ref{algo:RKMK}: the
evaluation of the map $\Lambda_i:=\dcay_{\mmi}^{-1}(\Lcal(\umi))$ and the computation of $\umi=\cay(\mmi)U_m$,
for any $i=2,\ldots,s$ and with $U_m\in\Mcal$.
First we prove that $\rank{\Lambda_i}\leq 4\Nrh$.
Observe that each term $\{\Lcal(\umi)\}_{i=1}^s$,
with $\Lcal$ defined in \eqref{eq:GSFcal}, can be written as
$\Lcal(\umi) = \gamma_i \delta_i^\top$ where
\begin{equation}\label{eq:GUcd}
	\gamma_i  := \big[\Fcal(\umi) \,|\, {-\umi}\big]\in\R{\Nf}{4\Nrh},
		\qquad
	\delta_i  := \big[\umi \,|\, \Fcal(\umi)\big]\in\R{\Nf}{4\Nrh}.
\end{equation}
For $i=1$, $\Omega_m^1=0$ and, hence, $\Lambda_1=\Lcal(U_m^1)=\gamma_1 \delta_1^\top$ owing to \eqref{eq:GUcd}.
Using definition \eqref{eq:dcayinv}, it holds
\begin{equation*}
	\Lambda_i =\dcay_{\mmi}^{-1}(\gamma_i\delta_i^\top)
	= \bigg(\Idm_{\Nf} - \dfrac{\mmi}{2} \bigg)\gamma_i \delta_i^\top
			\bigg(\Idm_{\Nf} + \dfrac{\mmi}{2} \bigg)=:e_i f_i^\top,
			\quad \forall\, i\geq 2,
\end{equation*}
where $e_i, f_i\in\R{\Nf}{4\Nrh}$ are defined as
\begin{equation*}
		e_i := \bigg(\Idm_{\Nf} - \dfrac{\mmi}{2}\bigg)\gamma_i,\qquad
		f_i := \bigg(\Idm_{\Nf} + \dfrac{(\mmi)^\top}{2}\bigg)\delta_i .
\end{equation*}
Using Line~\ref{line:Mi} of Algorithm~\ref{algo:RKMK},
the rank of $\mmi$ can be bounded as
\begin{equation*}
	\rank{\mmi}=\mathrm{rank}\bigg(\dt\sum_{j=1}^{i-1}a_{i,j}\Lambda_i\bigg)
	\leq \sum_{j=1}^{i-1}\rank{\Lambda_i}\leq 4\Nrh (i-1),
\end{equation*}
and similarly $\rank{\Omega_{m+1}}\leq 4\Nrh \Ns$.
Since the cost to compute each factor $e_i,f_i$ is $O(\Nfh\Nrh\,\rank{\mmi})$,
the computation of all $\Lambda_i$, for $1\leq i\leq \Ns$, requires $O(\Nfh\Nrh^2\Ns^2)$
operations.
Furthermore, in view of Proposition \ref{prop:Nn2},
each $\umi$ can be computed with $O(\Nfh\Nrh^2 i) + O(\Nrh^3 i^3)$ operations.
Summing over the number $\Ns$ of stages of the Runge--Kutta scheme, the computational complexity of Algorithm~\ref{algo:RKMK} becomes $O(\Nfh\Nrh^2 \Ns^2) + O(\Nrh^3 \Ns^4)$.
\end{proof}}
\bl{In principle one can solve the evolution equation \eqref{eq:Nalgebr}
on the Lie algebra $\fr{g}_{\Nf}$ using
the matrix exponential as coordinate map instead of the Cayley transform, in the spirit of \cite{MK99}.
However, there is no significant gain in terms of computational cost, as shown in details in Appendix~\ref{sec:exp}.}

Although the computational complexity of
Algorithm~\ref{algo:RKMK} is linear in the full dimension,
it presents a suboptimal dependence on
the number $\Ns$ of stages of the RK scheme.
However, in practical implementations,
the computational complexity of Proposition \ref{prop:RKcay}
might prove to be pessimistic in $\Ns$, and
might be mitigated with techniques that exploit the structure of the operators involved.


In the following Section we improve the efficiency of the numerical approximation of \eqref{eq:dynNbasis}
by developing a scheme
which is structure-preserving and has a computational cost $O(\Nfh\Nrh^2\Ns)$,
namely only linear in the dimension $\Nfh$ of the full model and in the number $\Ns$ of RK stages.


\subsection{Tangent methods on the orthosymplectic matrix manifold}\label{sec:TangMeth}
In this Section we derive a tangent method
based on retraction maps for the numerical solution of the reduced basis evolution problem
\eqref{eq:dynNbasis}.
The idea of tangent methods is presented in \cite[Section 2]{CO02} and consists in
expressing any $U(t)\in\Mcal$ in a neighborhood of a given $Q\in\Mcal$, via
a smooth local map $\Rcal_Q:\TM{Q}\rightarrow \Mcal$, as
\begin{equation}\label{eq:Usigma}
	U(t) = \Rcal_Q(V(t)),\qquad V(t)\in \TM{Q}.
\end{equation}
Let $\Rcal_Q$ be the restriction of a smooth map $\Rcal$
to the fiber $\TM{Q}$ of the tangent bundle. Assume that
$\Rcal_Q$ is defined in some open ball around $0\in\TM{Q}$,
and $\rqv=Q$ if and only if $V\equiv 0\in\TM{Q}$.
Moreover, let $\dr_Q:T\TM{Q}\cong\TM{Q}\times\TM{Q}\longrightarrow \TM{}$ be the \bl{(right trivialized)} tangent of the map $\Rcal_Q$, \bl{\emph{cf.} definition \eqref{eq:dcoo}}.
Let us fix the first argument of $\dr_Q$ so that, for any $U, V\in\Mcal$,
the tangent map $\drq{U}:\TM{Q}\rightarrow\TM{\Rcal_Q(U)}$
is defined as
$\drq{U}(V) = \dr_Q(U,V)$. Assume that
the local rigidity condition $\drq{0}=\Id_{\TM{Q}}$ is satisfied.
Under these assumptions, $\Rcal$ is a retraction and,
instead of solving the evolution problem \eqref{eq:dynNbasis} for $U$, one can
derive the local behavior of $U$ in a neighborhood of $Q$ by
evolving $V(t)$ in \eqref{eq:Usigma} in the tangent space of $\Mcal$ at $Q$.
Indeed, using \eqref{eq:dynNbasis} we can derive an evolution equation for $V(t)$ as 
\begin{equation*}
	\dot{U}(t) = \drq{V(t)}(\dot{V}(t))=\Fcal\big(\Rcal_Q(V(t))\big).
\end{equation*}
By the continuity of $V$ and the local rigidity condition, the map
$\drq{V(t)}$ is invertible for sufficiently small $t$
(i.e., $V(t)$ sufficiently close to $0\in\TM{Q}$) and hence
\begin{equation}\label{eq:EvolTM}
	\dot{V}(t) = \left(\drq{V(t)}\right)^{-1}\Fcal\big(\Rcal_Q(V(t))\big).
\end{equation}
Since the initial condition is $U(t_0)=Q$ it holds $V(t_0)=0\in\TM{Q}$.

This strategy allows to solve the ODE \eqref{eq:EvolTM} on the tangent space $\TM{}$,
which is a linear space,
with a standard temporal integrator and then
recover the approximate solution on
the manifold $\Mcal$ via the retraction map as in \eqref{eq:Usigma}.
If the retraction map can be computed exactly, this approach yields, by construction, a structure-preserving discretization.
The key issue here
is to build a suitable smooth retraction $\Rcal:\TM{}\rightarrow \Mcal$ such that
its evaluation and the computation of the inverse of its tangent map
can be performed exactly at a
computational cost that depends only linearly on the dimension of the full model.

In order to locally solve the evolution problem \eqref{eq:EvolTM} on the tangent space to the manifold $\Mcal$
at a point $Q\in\Mcal$ we
follow a similar approach to the one proposed in \cite{CO03} for the solution of differential equations on the Stiefel manifold.
\bl{Observe that, for any $Q\in\Mcal$, the velocity field $\Fcal(Q)$ in \eqref{eq:Fcal}, which describes the flow of the reduced basis on the manifold $\Mcal$, belongs to the space $\TsM{Q}$ defined in \eqref{eq:HU}.}
We thus construct a retraction $\Rcal_Q:\TsM{Q}\rightarrow\Mcal$ as composition of three \bl{functions}: a linear map $\Upsilon_Q$ from the \bl{space $\TsM{Q}$} to the Lie algebra $\fr{g}_\Nf$
associated with the Lie group $\Gcal_{\Nf}$ acting on the manifold $\Mcal$,
the Cayley transform \eqref{eq:cay} as coordinate map from the Lie algebra to the Lie group and the group action
$\Lambda:\Gcal_{\Nf}\times\Mcal\rightarrow\Mcal$,
\begin{equation*}
	\Lambda(G,Q) = \Lambda_Q(G) = G Q, \qquad \Lambda_Q: \Gcal_{\Nf}\longrightarrow\Mcal,
\end{equation*}
that we take to be the matrix multiplication.
This is summarized
in the diagram below,
\begin{center}
\begin{tikzpicture}[scale=1.2]

\node at (0,2) {$\T{}{\Gcal_{\Nf}}$}; 
\node at (2,2) {$ \fr{g}_{\Nf}$}; 
\node at (4,2) {$ \Gcal_{\Nf}$}; 
\node at (2,0) {\bl{$\TsM{Q}$}}; 
\node at (4,0) {$\Mcal$}; 

\draw[<-,dashed] (0.45,2) -- (1.65,2);
\node[left] at (1.5,2.2) {$\dcay$}; 
\draw[->] (2.45,2) -- (3.65,2);
\node[left] at (3.4,2.2) {$\cay$}; 

\draw[->] (4,1.75) -- (4,0.25);
\node[left] at (4.6,1) {$\Lambda_Q$}; 

\draw[->,dashed] (0,1.8) -- (1.8,0.25);
\node[left] at (0.85,0.9) {$\mathrm{d}\Lambda_Q$}; 

\draw[->] (1.9,1.75) -- (1.9,0.25);
\node[left] at (1.9,1) {$\Psi_Q$}; 
\draw[->] (2.1,0.25) -- (2.1,1.75);
\node[right] at (2.1,1) {$\Upsilon_Q$};

\draw[->] (2.45,0) -- (3.65,0);
\node[right] at (2.85,0.2) {$\Rcal_Q$};  

\end{tikzpicture}
\end{center}
%
In more details, we take $\Upsilon_Q$ to be,
for each $Q\in\Mcal$, the linear map $\Upsilon_Q: \bl{\TsM{Q}\subset\TM{Q}}\rightarrow \fr{g}_{\Nf}$ such
that $\Psi_Q\circ \Upsilon_Q = \Id_{\bl{\TsM{Q}}}$, where
$\Psi_Q={\bl{\mathrm{d}}\Lambda_Q}_{\big|_{e}}\circ\dcay_{0}$,
\bl{and 
$\TM{Q}  = \{V\in\R{\Nf}{\Nr}:\;Q^\top V\in \fr{g}_{\Nr}\}.$
The space $\TM{Q}$} can be characterized as follows.
\begin{proposition}
Let $Q\in\Mcal$ be arbitrary. Then, $V\in\TM{Q}$ if and only if
\begin{equation*}
\exists\, \Theta\in\R{\Nf}{\Nr}\;\mbox{ with }\;Q^\top \Theta\in \fr{sp}(\Nr)
	\quad\mbox{such that }\quad V=(\Theta Q^\top-Q\Theta^\top)Q.
\end{equation*}
\end{proposition}
\begin{proof}
($\Longleftarrow$) Assume that $V\in\R{\Nf}{\Nr}$ is of the form 
$V=(\Theta Q^\top-Q\Theta^\top)Q$ for some $\Theta\in\R{\Nf}{\Nr}$ with $Q^\top \Theta\in \fr{sp}(\Nr)$.
To prove that $V\in\TM{Q}$, we verify that $Q^\top V\in \fr{g}_{\Nr}$.
Using the orthogonality of $Q$, and the assumption $Q^\top \Theta\in \fr{sp}(\Nr)$ results in
\begin{equation*}
\begin{aligned}
	& Q^\top V = Q^\top (\Theta Q^\top-Q\Theta^\top)Q = 
	- Q^\top (Q \Theta^\top-\Theta Q^\top)Q = -V^\top Q.\\
	& Q^\top V\J{\Nr} = (Q^\top \Theta -\Theta^\top Q)\J{\Nr} = -\J{\Nr}(\Theta^\top Q-Q^\top \Theta)=-\J{\Nr}V^\top Q.
\end{aligned}
\end{equation*}
($\Longrightarrow$) Let $V\in\TM{Q}$, i.e. $Q^\top V\in\fr{g}_{\Nr}$.
Let $\Theta:=V+Q\big(S-\frac{Q^\top V}{2}\big)$ with $S\in\sym(\Nr)\cap\fr{sp}(\Nr)$ arbitrary.
We first verify that $Q^\top \Theta\in \fr{sp}(\Nr)$.
Using the orthogonality of $Q$, the fact that $V\in\TM{Q}$ and $S\in\fr{sp}(\Nr)$ results in
\begin{equation*}
	Q^\top \Theta\J{\Nr}+\J{\Nr} \Theta^\top Q
		= \dfrac{Q^\top V}{2}\J{\Nr} + \J{\Nr}\dfrac{V^\top Q}{2}+S\J{\Nr}+\J{\Nr} S^\top
		= S\J{\Nr}+\J{\Nr} S^\top =0.
\end{equation*}
We then verify that, with the above definition of $\Theta$, the matrix $(\Theta Q^\top-Q\Theta^\top)Q=\Theta - Q\Theta^\top Q$ coincides with $V$.
Using the fact that $S\in\sym(\Nr)$ and $V\in\TM{Q}$ yields
\begin{equation}\label{eq:idF}
\begin{aligned}
	\Theta - Q\Theta^\top Q
		& = V +	QS - Q\dfrac{Q^\top V}{2} - Q V^\top Q - Q\bigg(S^\top-\dfrac{V^\top Q}{2}\bigg)\\
		& = V - Q\dfrac{Q^\top V}{2} - Q\dfrac{V^\top Q}{2}
		 = V.
\end{aligned}
\end{equation}
\end{proof}
We can therefore characterize the tangent space of the orthosymplectic matrix manifold 
as
\begin{equation*}
\begin{aligned}
	\TM{Q}=\{V&\, \in\R{\Nf}{\Nr}:\;  V=(\Theta_Q^S(V) Q^\top-Q\Theta_Q^S(V)^\top)Q,\\
		& \;\mbox{with}\;
		\Theta_Q^S(V):=V+Q\bigg(S-\frac{Q^\top V}{2}\bigg),\;\mbox{for}\,S\in\sym(\Nr)\cap\fr{sp}(\Nr)\}.
\end{aligned}
\end{equation*}
This suggests that the linear map $\Upsilon_Q$ can be defined as
\begin{equation}\label{eq:defA}
\begin{array}{lcll}
	\Upsilon_Q: & \bl{\TsM{Q}} & \longrightarrow & \fr{g}_\Nf,\\
	     & V & \longmapsto & \Theta_Q^S(V) Q^\top - Q \Theta_{Q}^S(V)^\top.
\end{array}
\end{equation}
Indeed, since ${\bl{\mathrm{d}}\Lambda_{Q}}_{\big|_{e}}(G) = GQ$ and $\dcay_{0}=\Idm$, it holds
$(\Psi_Q\circ \Upsilon_Q)(V)= \Upsilon_Q(V)Q = V$ for any $V\in\bl{\TsM{Q}}$.
This stems from the definition of $\Upsilon_Q$ in \eqref{eq:defA} since
\begin{equation*}
\begin{aligned}
	\Psi_Q(\Upsilon_Q(V)) & = \big({\Lambda_{Q}}_{\big|_{e}}\circ\dcay_{0}\circ \Upsilon_Q\big)(V)
			  = {\Lambda_{Q}}_{\big|_{e}}(\Upsilon_Q(V))\\
			& = \Upsilon_Q(V)Q = \big(\Theta_Q^S(V) Q^\top - Q \Theta_Q^S(V)^\top\big)\,Q = V,
\end{aligned}
\end{equation*}
where the last equality follows by \eqref{eq:idF}.	
Note that $\Psi_Q={\mathrm{d}\Lambda_Q}_{\big|_{e}}\circ\dcay_{0}$
is not injective as $\Upsilon_Q(\bl{\TsM{Q}})$ is a proper subspace of $\fr{g}_\Nf$.
\bl{Observe that, for any $V\in\TM{Q}$, it holds $\Upsilon_Q(V)=VQ^\top-QV^\top+QV^\top QQ^\top$ and, hence, $\Upsilon_Q(V)\in\fr{g}_{\Nf}$.}
\begin{proposition}
Let $\cay:\fr{g}_\Nf\rightarrow \Gcal_{\Nf}$ be the Cayley transform defined in \eqref{eq:cay}.
For any $Q\in\Mcal$ and $S\in\sym(\Nr)\cap\fr{sp}(\Nr)$, we define
\begin{equation*}
\begin{array}{lcll}
	\Theta^S_Q: & \TM{Q} & \longrightarrow & \T{Q}{\Sp(\Nr,\r{\Nf})}=\{M\in\R{\Nf}{\Nr}:\;Q^\top M\in\fr{sp}(\Nr)\}\\
	     & V & \longmapsto & V+Q\bigg(S- \dfrac12 Q^\top V\bigg).
\end{array}
\end{equation*}
Then the map $\Rcal_Q:\bl{\TsM{Q}}\rightarrow \Mcal$ defined for any $V\in\bl{\TsM{Q}}$ as
\begin{equation}\label{eq:retraction}
	\Rcal_Q(V) = \cay(\Theta^S_Q(V)Q^\top-Q \Theta^S_Q(V)^\top) Q,
\end{equation}
is a retraction.
\end{proposition}
\begin{proof} We follow \cite[Proposition 2.2]{CO03}.
Let $V=0\in \TM{Q}$, then $\Theta^S_Q(0) = QS$ and then, using the fact that $S\in\sym(\Nr)$
and $\cay(0)=\Idm_\Nf$, it holds
$\Rcal_Q(0) = \cay\big(Q(S-S^\top)Q^\top\big)Q = \cay(0)Q=Q$.

Let $\Upsilon_Q$ be defined as in \eqref{eq:defA}.
Since, by construction $\Upsilon_Q$ admits left inverse it is injective and then $\Upsilon_Q(V)=0$ if and only if $V=0\in\bl{\TsM{Q}}$.
Then, $\Rcal_Q(V)=Q$ if and only if $\cay(\Upsilon_Q(V)) =\Idm_\Nf$, which implies $V=0\in\bl{\TsM{Q}}$.
%
Moreover, since $\Rcal_Q = \Lambda_Q\circ \cay\circ \Upsilon_Q$, the definition of group action
and the linearity of $\Upsilon$ result in
$\drq{0}=\Psi_Q\circ \Upsilon_Q = \Id_{\bl{\TsM{Q}}}.$
\bl{It can be easily verified that $\Rcal_Q(V)\in\Mcal$ for any $V\in\TsM{Q}$.}
\end{proof}
Note that the matrix $S\in\sym(\Nr)\cap\fr{sp}(\Nr)$ in the definition of the retraction \eqref{eq:retraction}
is of the form
\begin{equation*}
S = 
\begin{pmatrix}
A & B\\
B & -A	
\end{pmatrix},\qquad\qquad\mbox{with}\quad A, B\in \sym(\Nrh).
\end{equation*}
Its choice affects the numerical performances of the algorithm for the computation of the retraction and its inverse tangent
map, as pointed out in \cite[Section 3]{CO03}.

In the following Subsections we propose a temporal discretization
of \eqref{eq:EvolTM} with
an $\Ns$-stage explicit Runge--Kutta method and show that the resulting algorithm
has arithmetic complexity of order $C_{\Fcal}+O(\Nfh\Nrh^2)$ at every stage of the temporal solver.

\subsubsection{Efficient computation of retraction and inverse tangent map}
In the interval $(t^m,t^{m+1}]$ the local evolution on the tangent space,
corresponding to \eqref{eq:EvolTM}, reads
\begin{equation*}
	\dot{V}(t) = \bigg({\dr_{U_m}}_{\big|_{V(t)}}\bigg)^{-1}\Fcal\big(\Rcal_{U_m}(V(t))\big) =:f_m(V(t)).
\end{equation*}
Let $(b_i,a_{i,j})$ for $i=1,\ldots,s$ and $j=1,\ldots,i-1$
be the coefficients of the Butcher tableau describing the $\Ns$-stage explicit Runge--Kutta method.
Then the numerical approximation of \eqref{eq:EvolTM}-\eqref{eq:Usigma} with
$U_0:=Q\in\Mcal$ and $V_0=0\in\TM{Q}$ is given in Algorithm~\ref{algo:RKMKtang}.
\begin{algorithm}
\caption{Explicit RK-tangent scheme in $(t^m,t^{m+1}]$}\label{algo:RKMKtang}
\bl{\normalsize
\begin{algorithmic}[1]
 \Require{$U_m\in\Mcal$, $\{b_i\}_{i=1}^{\Ns}$, $\{a_{i,j}\}_{i,j=1}^{\Ns}$ }
 \State $A_m^1=\Fcal(U_m)$
 \For{$i=2,\ldots,\Ns$}
 \State $A_m^i = f_m\bigg(\dt\sum\limits_{j=1}^{i-1} a_{i,j}A_m^j\bigg)$,
 \EndFor
 \State $V_{m+1} = \dt\sum\limits_{i=1}^\Ns b_iA^i_m$,
 \State \Return $U_{m+1} = \Rcal_{U_m}(V_{m+1})\in\Mcal$
\end{algorithmic}}
\end{algorithm}
%
Other than the evaluation of the velocity field $\Fcal$ at 
$\Rcal_{U_m}(V)$,
the crucial points of Algorithm~\ref{algo:RKMKtang} in terms of computational cost,
are the evaluation of the retraction and the computation of its inverse tangent map.
If we assume that both operations can be performed with a computational cost of order $O(\Nfh\Nrh^2)$,
then Algorithm~\ref{algo:RKMKtang} has an overall
arithmetic complexity of order $O(\Nfh\Nrh^2\Ns)+C_{\Fcal}\Ns$,
where $C_{\Fcal}$ is the cost to compute $\Fcal(U)$ in \eqref{eq:Fcal} at any given $U\in\Mcal$.

\medskip
\noindent
\textbf{Computation of the retraction.}
A standard algorithm to compute the retraction $\Rcal_Q$ \eqref{eq:retraction} at the matrix $V\in\R{\Nf}{\Nr}$
requires $O(\Nfh^2\Nrh)$ for the multiplication between $\cay(\Upsilon_Q(V))$ and $Q$,
plus the computational cost to evaluate the Cayley transform at $\Upsilon_Q(V)\in\R{\Nf}{\Nf}$.
However, 
\bl{for any $V\in\TsM{Q}$,}
the matrix $\Upsilon_Q(V)\in \fr{g}_{\Nf}$
admits the low-rank splitting
%
\begin{equation*}
	\Upsilon_Q(V) = \Theta^S_Q(V)Q^\top - Q \Theta^S_Q(V)^\top = \alpha\beta^\top,
\end{equation*}
where
\begin{equation}\label{eq:AB}
	\alpha := \big[\,\Theta^S_Q(V) \,|\, {-Q}\,\big]\in\R{\Nf}{4\Nrh},\qquad
	\beta  := \big[\,Q \,|\, \Theta^S_Q(V)\,\big]\in\R{\Nf}{4\Nrh}.
\end{equation}
We can revert to the results of Proposition \ref{prop:Nn2} (with $\Nm=4\Nrh$) so that
the retraction \eqref{eq:retraction} can be computed as
\begin{equation*}
	\Rcal_Q(V)=\cay(\Upsilon_Q(V)) Q
			 = Q + \alpha (\beta^\top\alpha)^{-1}\big(\cay(\beta^\top\alpha)-\Idm_{4\Nrh}\big)\beta^\top Q,
\end{equation*}
with computational cost of order $O(\Nfh\Nrh^2)$.

\medskip
\noindent
\textbf{Computation of the inverse tangent map of the retraction.}
Let $Q\in\Mcal$ and $V\in\bl{\TsM{Q}}$. 
Using the definition of retraction \eqref{eq:retraction} we have
\begin{equation*}
	\rqv = \cay(\Upsilon_Q(V))Q = (\Lambda_Q\circ\cay\circ \Upsilon_Q)(V).
\end{equation*}
Then, the tangent map $\dr_Q$ reads
\begin{equation*}
	\dr_Q = \bl{\mathrm{d}}\Lambda_Q\circ\dcay\circ \bl{\mathrm{d}}\Upsilon_Q:T\bl{\TsM{Q}}\longrightarrow\T{}{\fr{g}_\Nf}\cong\fr{g}_\Nf\longrightarrow\T{}{\Gcal_{\Nf}}\longrightarrow \TM{Q}.
\end{equation*}
Fixing the fiber on $T\bl{\TsM{Q}}$ corresponding to $V\in\bl{\TsM{Q}}$
results in
\begin{equation*}
\begin{aligned}
	\drq{V}(\vt) & = \dr_Q(V,\vt)
		 = {\bl{\mathrm{d}}\Lambda_Q}_{\big|_{\cay(\Upsilon_Q(V))}}\circ \dcay_{\Upsilon_Q(V)}(\Upsilon_Q(\vt))\\
		& = \dcay_{\Upsilon_Q(V)}(\Upsilon_Q(\vt))\,\cay(\Upsilon_Q(V))Q = \dcay_{\Upsilon_Q(V)}(\Upsilon_Q(\vt))\,\rqv,
\end{aligned}
\end{equation*}
where we have used the linearity of the map $\Upsilon_Q$.

Assume we know $W\in\bl{\TsM{\rqv}}$.
We want to compute $\vt\in\bl{\TsM{Q}}$ such that
\begin{equation}\label{eq:invR}
	\drq{V}(\vt) = \dcay_{\Upsilon_Q(V)}(\Upsilon_Q(\vt))\,\rqv = W.
\end{equation}

It is possible to solve problem \eqref{eq:invR} with
arithmetic complexity $O(\Nfh\Nrh^2)$ by proceeding as in
\cite[Section 3.2.1]{CO03}.
Since, for our algorithm, the result of \cite{CO03} can be extended to the case of arbitrary
matrix $S\in\sym(\Nr)\cap\fr{sp}(\Nr)$ in \eqref{eq:retraction},
we report the more general derivation in Appendix \ref{app:RtangInv}.
\bl{Note that, for $S=0$ and explicit Euler scheme, the two Algorithms \ref{algo:RKMK} and
\ref{algo:RKMKtang} are equivalent.}

\subsubsection{Convergence estimates for the tangent method}
Since the retraction and its inverse tangent map in Algorithm~\ref{algo:RKMKtang} can be computed exactly,
the smoothness properties of $\Rcal$ allow to derive error estimates for the
approximate reduced basis in terms of the numerical solution of the evolution problem \eqref{eq:EvolTM}
in the tangent space.
\begin{proposition}
The retraction map $\Rcal:\TM{}\rightarrow\Mcal$ defined in \eqref{eq:retraction}
is locally Lipschitz continuous in the Frobenius $\norm{\cdot}$-norm, namely for any $Q\in\Mcal$,
$\Rcal_Q:\bl{\TsM{Q}}\rightarrow\Mcal$ satisfies
\begin{equation*}
	\norm{\Rcal_Q(V)-\Rcal_Q(W)}\leq 3\norm{V-W},
	\qquad\forall\, V,\, W\in\bl{\TsM{Q}}.
\end{equation*}
\end{proposition}
\begin{proof}
Let $U := \Rcal_Q(V) = \cay(\Upsilon_Q(V))Q$ and
$Y :=\Rcal_Q(W) = \cay(\Upsilon_Q(W))Q$.
Using the definition of Cayley transform \eqref{eq:cay}
we have, for $\overline{\Upsilon}_Q(\cdot):=\Upsilon_Q(\cdot)/2$,
\begin{equation*}
\begin{aligned}
0	& = \big(\Idm_{\Nf}-\aqvb\big)U     - \big(\Idm_{\Nf}-\overline{\Upsilon}_Q(W)\big)Y 
	 - \big(\Idm_{\Nf}+\aqvb\big)Q     - \big(\Idm_{\Nf}+\overline{\Upsilon}_Q(W)\big)Q\\
	& = \big(\Idm_{\Nf}-\aqvb\big)(U-Y) - \big(\aqvb-\overline{\Upsilon}_Q(W)\big)(Q+Y).
\end{aligned}
\end{equation*}
Since $\Upsilon_Q$ is skew-symmetric $\big(\Idm_{\Nf}-\aqvb\big)^{-1}$ is normal.
Then,
$$\norm{\big(\Idm_{\Nf}-\aqvb\big)^{-1}}_2=\rho\big[\big(\Idm_{\Nf}-\aqvb\big)^{-1}\big]\leq 1.$$
Hence, since $Q$ and $Y$ are (semi-)orthogonal matrices, it holds
\begin{equation*}
\norm{U-Y} \leq \norm{\big(\Idm_{\Nf}-\aqvb\big)^{-1}}_2\norm{\Upsilon_Q(V)-\Upsilon_Q(W)}
		\leq \norm{\Upsilon_Q(V)-\Upsilon_Q(W)}.
\end{equation*}
Using the definition of $\Upsilon_Q$ from \eqref{eq:defA} results in
\begin{equation*}
\norm{\Upsilon_Q(V)-\Upsilon_Q(W)}  = \norm{(V-W)Q^\top - Q(V-W) + Q(V^\top-W^\top)QQ^\top}
	 \leq 3\norm{V-W}.
\end{equation*}
\end{proof}
It follows that the solution of Algorithm~\ref{algo:RKMKtang}ref{eq:cayRKtang} can be computed with the same order of
accuracy of the RK temporal scheme.
\begin{corollary}
For $Q\in\Mcal$ given, let $\Rcal_Q$ be the retraction map defined in \eqref{eq:retraction}.
Let $U(t^m)=\Rcal_Q(V(t^m))$, where
$V(t^m)$ is the exact solution of \eqref{eq:EvolTM} at a given time $t^m$
and let $U_m=\Rcal_Q(V_m)$, where $V_m$ is the numerical solution of \eqref{eq:EvolTM} at time $t^m$
obtained with Algorithm~\ref{algo:RKMKtang}.
Assume that the numerical approximation of the evolution equation for the unknown $V$ on the tangent space of $\Mcal$
is of order $O(\dt^k)$.
Then, it holds
\begin{equation*}
	\norm{U(t^m)-U_m} =O(\dt^k).
\end{equation*}
\end{corollary}

\bl{
\section{Numerical experiment}\label{sec:numExp}
To gauge the performances of the proposed 
method, we consider the numerical simulation of the
finite-dimensional parametrized Hamiltonian system arising from the spatial approximation of the one-dimensional shallow water equations (SWE).
The shallow water equations are used in oceanography to describe the kinematic behavior of thin inviscid single fluid layers flowing over a changing topography.
Under the assumptions of irrotational flow and flat bottom topography, the fluid is described by
the scalar potential $\phi$ and the height $h$ of the free-surface, normalized by its mean value, via the nonlinear system of PDEs
\begin{equation}\label{eq:SWE}
    \left\{
    \begin{aligned}
    & \partial_t h + \partial_x (h\,\partial_x \phi)=0, &\qquad \mbox{in}\;(-L,L)\times (0,T],\\
    & \partial_t \phi + \dfrac12|\partial_x \phi|^2 + h =0, &\qquad \mbox{in}\;(-L,L)\times (0,T],
    \end{aligned}\right.
\end{equation}
where $L=10$, $T=7$, $h,\phi:[-L,L]\times(0,T]\times\Sprm\rightarrow \mathbb{R}$ are the state variables, and $\Sprm\subset\r{2}$ is a compact set of parameters.
Here we consider $\Sprm := \left [0.1,0.15 \right ] \times \left [ 0.2,1.5 \right ]$.
The system is provided with periodic boundary conditions for both state variables, and with parametric initial conditions
$(h^{0}(x;\prm),\phi^{0}(x;\prm)) = (1+\alpha e^{-\beta x^2},0)$,
where $\alpha$ controls the amplitude of the initial hump in the depth, $\beta$ describes its width, and $\prm=(\alpha,\beta)$.

For the numerical discretization in space, we consider
a Cartesian mesh on $[-L,L)$ with $N{-}1$ equispaced intervals
and we denote with $\Delta x$ the mesh width.
The degrees of freedom of the problem are the
nodal values of the height and potential, i.e.
$(h_h(t;\prm),\phi_h(t;\prm))=(h_1,\dots,h_N,\phi_1,\dots,\phi_N)$.
The discrete set of parameters $\Sprmh$ is obtained by uniformly sampling $\Sprm$ with $10$ samples per dimension, for a total of $\Np=100$ different configurations.
This implies that the full model variable $\Rcal$
in \eqref{eq:dynNR} is the $\Nfh\times\Np$ matrix
given by
$\Rcal_{i,k}(t)=h_i(t;\prmh^k)$ if $1\leq i\leq \Nfh$ and
$\Rcal_{i,k}(t)=\phi_{i-N}(t;\prmh^k)$ if $\Nfh{+}1\leq i\leq \Nf$,
for any $k\in\{1,\ldots,\Np\}$,
where $\prmh^k$ denotes the $k$-th entry of the vector $\prmh\in\r{\Np}$ containing the samples of the parameters.
We consider second order accurate centered finite difference schemes to discretize the first order spatial derivative in \eqref{eq:SWE}.
The evolution problem \eqref{eq:SWE}
admits a canonical symplectic Hamiltonian. Spatial discretization with centered finite differences yields a Hamiltonian dynamical system where the Hamiltonian
associated with the $k$-th parameter is given by
\begin{equation*}
    \HamN_k(\Rcal(t)) = \dfrac12 \sum_{i=1}^{N}
    \bigg(h_i(t;\prmh^k) \left( \dfrac{\phi_{i+1}(t;\prmh^k)-\phi_{i-1}(t;\prmh^k)}{2\Delta x} \right)^2 + h_{i}^2(t;\prmh^k)\bigg).
\end{equation*}
Wave-type phenomena often exhibit a low-rank behavior only locally in time, and, hence, global (in time) model order reduction proves ineffective in these situations. We show this behavior by comparing the performances of our dynamical reduced basis method with the global symplectic reduced basis approach of \cite{PM16} based on complex SVD.
For the latter, a symplectic reduced space is obtained from the full model obtained by discretizing \eqref{eq:SWE} with centred finite differences in space, with $\Nfh=1000$, and the implicit midpoint rule in time, with $\dt=10^{-3}$. We consider snapshots every $10$ time steps and
$4$ uniformly distributed samples of $\Sprm$ per dimension.
Concerning the dynamical reduced model, we evaluate the initial condition $(h_h(0;\prmh),\phi_h(0;\prmh))$ at all values $\prmh$ and compute the matrix $\Rcal_0\in\mathbb{R}^{\Nf\times \Np}$ having as columns each of the evaluations. As initial condition for the reduced system \eqref{eq:eqU}, we use
$U(0) = U_0\in\mathbb{R}^{\Nf\times \Nr}$ obtained via complex SVD of the matrix $\Rcal_0$ truncated at $\Nrh$, while
$Z(0) = U_0^T\Rcal_0$. Then,
we solve system \eqref{eq:eqU} with
a 2-stage partitioned Runge-Kutta method obtained as follows:
the evolution equation for the coefficients $Z$ is discretized with the implicit midpoint rule; while the evolution equation \eqref{eq:dynNbasis} for the reduced basis
is solved using the tangent method described in Algorithm~\ref{algo:RKMKtang}
with the explicit midpoint scheme, i.e. $\Ns=2$,
$b_1=0$, $b_2=1$, and  $a_{1,1}=a_{1,2}=a_{2,2}=0$, $a_{2,1} = 1/2$.
Note that the resulting partitioned RK method has order of accuracy 2 and the numerical integrator for $Z$ is symplectic \cite[Section III.2]{HaLuWa06}.
Finally, the nonlinear quadratic operator in \eqref{eq:SWE}, is reduced by using tensorial techniques \cite{SSN14}.

In Figure~\ref{fig:error_SWE1D} we report the error in the Frobenius norm, at final time, between the full model solution and the reduced solution obtained with the two different approaches and various dimensions of the reduced space.
Note that the runtime includes also the offline phase for the global approach.}
%
%

\begin{figure}[H]
\centering
\begin{tikzpicture}[scale=0.9]
    \begin{axis}[xlabel= {runtime $[s]$},
    			 ylabel={$\norm{\Rcal(T)-R(T)}$},
                 grid=minor,
                 xmode=log,
                 ymode=log,
                 xmax = 11000,
                 ymin=0.00003,
                 ymax = 0.04,
                 xlabel style={font=\normalsize},
                 ylabel style={at={(0.04,0.5)},font=\normalsize},
                 x tick label style={font=\small},
                 y tick label style={font=\small}]
        \addplot+[mark=*,color=blue,mark size=2pt,
            every node near coord/.append style={xshift=0.65cm},
            every node near coord/.append style={yshift=-0.2cm},
            nodes near coords, 
            point meta=explicit symbolic,
            every node near coord/.append style={font=\small}] table[x=Timing,y=Final_Error, meta index=2]
            {DynRed.txt};
        \addplot+[mark=square*,color=red,mark size=2pt,
            every node near coord/.append style={xshift=0.5cm},
            every node near coord/.append style={yshift=0.05cm},
            nodes near coords, 
            point meta=explicit symbolic,
            every node near coord/.append style={font=\small}] table[x=Timing,y=Final_Error, meta index=2]
            {GloRed_method_Complex_SVD.txt};
        \addplot+[mark=none,color=black,very thick,dashed]
table[x=Timing,y=DummyError]
            {Full.txt}; 
         \legend{Dynamical RBM,
                Global RBM,
                Full model};
    \end{axis}
\end{tikzpicture}
\caption{Error between the full model solution and the reduced solution at final time vs. the algorithm runtime.}
\label{fig:error_SWE1D}
\end{figure}
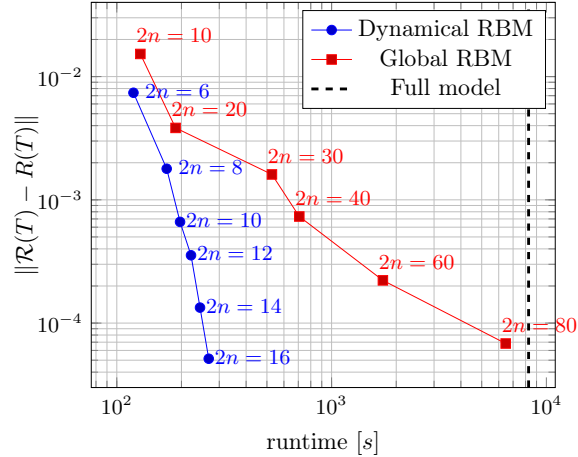
\bl{The results of Figure~\ref{fig:error_SWE1D} show that the dynamical reduced basis method outperforms the global approach by reaching comparable accuracy at a reduced computational cost. Moreover,
as the dimension of the reduced space increases, the runtime of the global method becomes comparable to the one required to solve the high-fidelity problem, meaning that there is no gain in performing global model order reduction.  

Figure~\ref{fig:error_Ham_SWE1D} shows the evolution of the error in the conservation of the discrete Hamiltonian, averaged over all $\Np$ values of the parameter.
Since the Hamiltonian is a cubic quantity, we do not expect exact conservation associated with the proposed partitioned RK scheme. In addition,
as pointed out at the end of Section~\ref{sec:SDLR},
we cannot guarantee exact preservation of the invariants at the interface between temporal intervals, since the reduced solution is projected into the
space spanned by the updated basis.
However, the preservation of the symplectic structure both in the reduction and in the discretization yields a good control on the Hamiltonian error, as it can be observed in Figure \ref{fig:error_Ham_SWE1D}.}

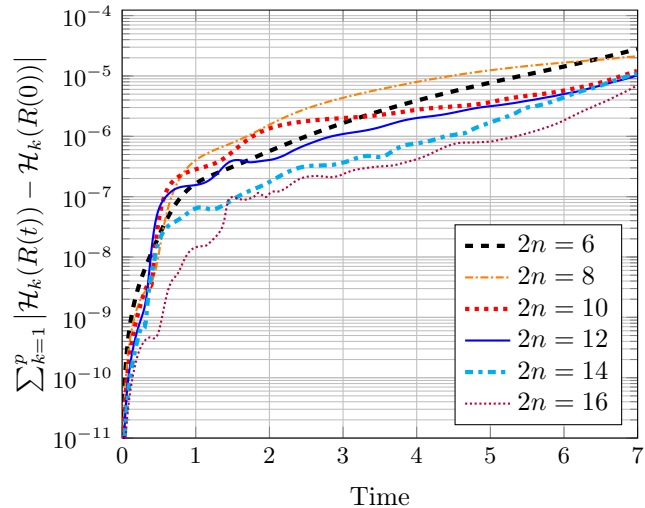
\begin{figure}[H]
\centering
\begin{tikzpicture}[scale=1]
    \begin{axis}[xlabel= {Time},
    			 ylabel={$\sum_{k=1}^{\Np}\big|\Hcal_k(R(t))-\Hcal_k(R(0))\big|$},
                 grid=both,
                 ymode=log,
                 ymin = 1e-11,
                 xmin = 0,
                 xmax = 7,
                 xlabel style={font=\normalsize},
                 ylabel style={at={(0.001,0.5)},font=\normalsize},
                 x tick label style={font=\small},
                 y tick label style={font=\small},
                 legend pos = south east]
\addplot+[mark=none,color=black,ultra thick,dashed] table[x=Timing,y=Error] {dynROM6.txt};
\addplot+[mark=none,color=orange,thick,densely dashdotted] table[x=Timing,y=Error] {dynROM8.txt};
\addplot+[mark=none,color=red,ultra thick,dotted] table[x=Timing,y=Error] {dynROM10.txt};
\addplot+[mark=none,color=blue, thick,solid] table[x=Timing,y=Error] {dynROM12.txt};
\addplot+[mark=none,color=cyan,ultra thick,dashdotted] table[x=Timing,y=Error] {dynROM14.txt};
\addplot+[mark=none,color=purple,thick,densely dotted] table[x=Timing,y=Error] {dynROM16.txt};
         \legend{$\Nr=6\phantom{1}$,$\Nr=8\phantom{1}$,$\Nr=10$,
         $\Nr=12$,
         $\Nr=14$,
         $\Nr=16$};
    \end{axis}
\end{tikzpicture}
\caption{Evolution of the error in the conservation of the Hamiltonian.}
\label{fig:error_Ham_SWE1D}
\end{figure}

\section{Concluding remarks and future work}\label{sec:conclusions}
Nonlinear dynamical reduced basis methods
for parameterized finite-dimensional Hamiltonian systems
have been developed \bl{to mitigate the computational burden of
large-scale, multi-query and long-time simulations.
The proposed} techniques provide an attractive computational
approach to deal with the local low-rank nature of Hamiltonian dynamics
while preserving the geometric structure of the phase space
even at the discrete level.

Possible extensions of this work involve
the numerical study of the proposed
algorithm including high order splitting temporal integrators,
numerical approximations ensuring the
exact conservation of Hamiltonian, and
restarting procedures of the Cayley RK algorithm.
%
Moreover, the extension of dynamical reduced basis methods to Hamiltonian systems
with a nonlinear Poisson structure would allow nonlinear structure-preserving
model order reduction of a large class of problems, including
Euler and Vlasov--Maxwell equations.
Some of these topics will be investigated in forthcoming works.

\begin{acknowledgment*}
The author gratefully acknowledges many fruitful discussions with Jan S. Hesthaven,
and would like to thank Elena Celledoni for pointing out references \cite{CO02,CO03}.
\end{acknowledgment*}

\section*{References}\addcontentsline{toc}{section}{References}
{
\small
\printbibliography[heading=none] 
}

\appendix
\section{Exponential map}\label{sec:exp}

\bl{Let us consider Algorithm~\ref{algo:RKMK} with the exponential as coordinate map, namely $\coo=\exp:\fr{g}_{\Nf}\rightarrow\fr{g}_{\Nf}$.
For any $\mmi\in\fr{g}_{\Nf}$ and $\umi\in\Mcal$, the matrix $\dexp_{\mmi}^{-1}(\Lcal(\umi))$ in Line~\ref{line:Mi} can be approximated by truncating the Baker--Campbell--Hausdorff (BCH) formula as
\begin{equation}\label{eq:BCHapp}
\dexp_{\mmi}^{-1}(\Lcal(\umi))\approx
\widehat{\Lambda}_i:= 
\sum_{k=0}^{q} \dfrac{\mathcal{B}_k}{k!}\,\ad{\mmi}^{k}(\Lcal(\umi)),\quad\mbox{for some } q\in\mathbb{N},
\end{equation}
where $\mathcal{B}_k$ denotes the $k$-th Bernoulli number and $\ad{\mmi}^0(\Lcal(\umi))=\Lcal(\umi)$.
Observe that $\widehat{\Lambda}_i\in\R{\Nf}{\Nf}$ belongs to the Lie algebra $\fr{g}_{\Nf}$, and, hence, each $\mmi$ is in $\fr{g}_{\Nf}$ and the solution of the RK-MK method remains on $\Mcal$, see e.g. \cite[Theorem 8.4]{HaLuWa06}.

To assess the computational complexity of Algorithm~\ref{algo:RKMK},
we need to consider
two operations: the
evaluation of $\widehat{\Lambda}_i$ and the computation of $\umi=\exp(\mmi)U_m$,
for any $i=2,\ldots,\Ns$. To this aim, we first rewrite each commutator in \eqref{eq:BCHapp} as a matrix polynomial.
\begin{lemma}\label{lem:pol}
Let $A,B\in\R{\Nf}{\Nf}$. For any fixed $k\geq 0$, there exist coefficients $\{c^{(k)}_h\}_{h=0}^k\subset\mathbb{R}$ such that
\begin{equation}\label{eq:pol}
	\ad{A}^k(B) = \sum_{h=0}^k c_h^{(k)} A^h B A^{k-h}.
\end{equation}
\end{lemma}
\begin{proof}
We proceed by induction on $k$. For $k=0$, $\ad{A}^0(B)=B$ and $c^{(0)}_0=1$.
For $k=1$, $\ad{A}(B)=[A,B]=AB-BA=c^{(1)}_0 BA + c^{(1)}_1 AB$, with
$c^{(1)}_0=-1$ and $c^{(1)}_0=1$.
Assume that $\ad{A}^{k-1}(B)$, with $k\geq 2$, can be expressed in polynomial form. Then,
\begin{equation*}
\begin{aligned}
	\ad{A}^k(B) & = [A,\ad{A}^{k-1}(B)]
	= \sum_{h=0}^{k-1} c_h^{(k-1)} A^{h+1} B A^{k-1-h} - \sum_{h=0}^{k-1} c_h^{(k-1)} A^{h} B A^{k-h},
\end{aligned}
\end{equation*}
which is of the form \eqref{eq:pol}
with $c_0^{(k)}=-c_0^{(k-1)}$, $c_k^{(k)}=c_{k-1}^{(k-1)}$ and
$c_h^{(k)}=c_{h-1}^{(k-1)}-c_h^{(k-1)}$ for any $1\leq h\leq k-1$.
\end{proof}
The matrix polynomial form \eqref{eq:pol} allows to estimate the rank of the $\{\widehat{\Lambda}_i\}_{i=1}^{\Ns}$.
\begin{lemma}\label{lem:rAi}
Let $\widehat{\Lambda}_i\in\fr{g}_{\Nf}$ be defined as in \eqref{eq:BCHapp}.
Then,
\begin{equation}
	\rank{\widehat{\Lambda}_i}\leq \min\{2^{i-1}, q+1\}\,\rank{\Lcal(\umi)}.
\end{equation} 
\end{lemma}
\begin{proof}
Using Lemma~\ref{lem:pol}, we have that, for any $A,B\in\fr{g}_{\Nf}$,
\begin{equation*}
	C:=\sum_{k=0}^q \ad{A}^{k}(B) = \sum_{k=0}^q\sum_{h=0}^k c_h^{(k)} A^h B A^{k-h}
	= \sum_{h=0}^q A^h B \sum_{k=h}^q  c_h^{(k)} A^{k-h}.
\end{equation*}
Since the rank of a matrix product is bounded by the minimum among the ranks of the factors, this implies that
$\rank{C}\leq (q+1)\rank{B}$ and, hence,
$\rank{\widehat{\Lambda}_i}\leq (q+1)\, \rank{\Lcal(\umi)}$.

We now prove that, for any $0\leq k\leq q$, there exists matrices $E_k, D_k\in\R{\Nf}{\Nf}$
such that $\ad{A}^{k}(B)=A E_k + BD_k$.
We proceed by induction on $k$.
For $k=0$, $\ad{A}^0(B)=B$ so that $E_0=0$ and $D_0=\Idm_{\Nf}$.
For $k=1$,
$\ad{A}(B)=[A,B]=AB-BA$ so that $E_1=B$ and $D_1=-A$.
Assume that the statement holds for $k-1$, with $k\geq 3$, then
$\ad{A}^{k}(B)=A\, \ad{A}^{k-1}(B)-\ad{A}^{k-1}(B)A = A(AE_{k-1}+BD_{k-1})-(AE_{k-1}+BD_{k-1})A
=AE_{k} + BD_{k}$
with $E_{k}=AE_{k-1}+BD_{k-1}-E_{k-1}A$ and $D_{k}=-D_{k-1}A$.
Therefore,
$$C=\sum_{k=0}^{q}(AE_{k} + BD_{k})=A\bigg(\sum_{k=1}^{q}E_{k}\bigg)+B\bigg(\sum_{k=1}^{q}D_{k}\bigg),$$
and, hence,
$\rank{C}\leq \rank{A}+\rank{B}$.
This is equivalent to
$\rank{\widehat{\Lambda}_i}\leq\rank{\mmi}+\rank{\Lcal(\umi)}$.
Using the definition of $\mmi$ from Line~\ref{line:Mi} of Algorithm~\ref{algo:RKMK},
the rank of $\widehat{\Lambda}_i$, for any $i\geq 2$, can be bounded as
\begin{equation*}
	\rank{\widehat{\Lambda}_i} \leq \rank{\Lcal(\umi)} +
	\sum_{j=1}^{i-1}\rank{\widehat{\Lambda}_j}.
\end{equation*}
Since $\rank{\widehat{\Lambda}_1}=\rank{\Lcal(U_m)}$, it easily follows by induction that
$\rank{\widehat{\Lambda}_i} \leq 2^{i-1}\,\rank{\Lcal(\umi)}$.
\end{proof}
Observe that the factorization \eqref{eq:GUcd} implies that each term $\{\Lcal(\umi)\}_{i=1}^s$, with $\Lcal$ defined in \eqref{eq:GSFcal}, has rank at most $4\Nrh$.
Therefore, from Lemma~\ref{lem:rAi}, it follows that
\begin{equation}\label{eq:rM}
r_i:= \rank{\mmi}\leq \sum_{j=1}^{i-1} \rank{\widehat{\Lambda}_j}\leq
4\Nrh \sum_{j=1}^{i-1} \min\{2^{i-1}, q+1\}.
\end{equation}
It can be inferred from \eqref{eq:rM} that the bound $q+1$ is the one dominating in the computation of $\mmi$ whenever the number $\Ns$ of RK stages is sufficiently large. 
An optimal number $q_{\opt}$ of commutators to achieve the accuracy of the corresponding RK method can be derived as in \cite{BCR02,CB03}.
We consider a few examples from
\cite[Table 3.1]{CB03}:
RKF45 has $\Ns=6$, $q_{\opt}=5$, and hence $q_{\opt}+1 \leq 2^{i-1}$ for $i\geq 4$; 
DVERK has $\Ns=8$, $q_{\opt}=10$, and hence $q_{\opt}+1 \leq 2^{i-1}$ for $i\geq 5$;
Butcher7 has $\Ns=9$, $q_{\opt}=21$, and hence $q_{\opt}+1 \leq 2^{i-1}$ for $i\geq 6$.
In light of these considerations, we consider the bound $r_i\leq 4\Nrh(i-1)(q+1)$
although for small $i$ this might not be sharp.
With the estimate \eqref{eq:rM} on the rank of $\mmi$, we can assess the cost of computing $\umi$ at each stage of the RK-MK Algorithm~\ref{algo:RKMK}.
The computation of the exponential of a matrix in $\fr{g}_{\Nf}$ requires $O(\Nfh^3)$ operations, but this cost can be mitigated whenever the argument of the exponential is of low-rank.
Similarly to Proposition \ref{prop:Nn2}, it can be shown that the cost to compute $\exp(uv^\top)Y$ with $u,v\in\R{\Nf}{k}$ and $Y\in\R{\Nf}{\Nr}$ is $O(k^3+k^2\Nrh+\Nfh \Nrh k)$ \cite[Proposition 3]{CI00}.
In Algorithm~\ref{algo:RKMK} we need to evaluate the exponential of $\{\mmi\}_{i=1}^s$,
with $\rank{\mmi}\leq 4\Nrh (i-1)(q+1)$.
Therefore, the computation of all $\umi$ in Line~\ref{line:Ui}, for $2\leq i\leq \Ns$, requires
$O(\Nfh \Nrh^2 s^2 q+\Nrh^3 \Ns^{4} q^3)$.


The other contribution to the computational cost of Algorithm~\ref{algo:RKMK} comes from the evaluation of each $\widehat{\Lambda}_i$ in \eqref{eq:BCHapp}. To estimate this cost, we resort to the polynomial expression \eqref{eq:pol} and the low-rank splitting of $\mmi = \alpha_i\beta_i^\top$, with
$\alpha_i, \beta_i\in\R{\Nf}{r_i}$, and of $\Lcal(\umi)=\gamma_i \delta_i^\top$
with $\gamma_i, \delta_i\in\R{\Nf}{4\Nrh}$.
Let $\hat{c}_j^{(k)}:=c_j^{(k)} \Bcal_k/k!$, for any $0\leq k,j \leq q$, then
\begin{equation*}
\begin{aligned}
	\widehat{\Lambda}_i &= \hat{c}^{(0)}_0 \gamma_i\delta_i^\top
	+ \sum_{j=1}^q \big( \hat{c}_0^{(j)} \gamma_i\delta_i^\top (\alpha_i\beta_i^\top)^j + \hat{c}_j^{(j)} (\alpha_i\beta_i^\top)^j \gamma_i\delta_i^\top \big) \\
	&\quad + \sum_{j=1}^q \alpha_i (\beta_i^\top\alpha_i)^{j-1}\beta_i^\top \gamma_i \delta_i^\top \alpha_i\bigg(\sum_{k=j+1}^q  \hat{c}_j^{(k)} (\beta_i^\top\alpha_i)^{k-j-1}\bigg)\beta_i^\top\\
	& = \hat{c}^{(0)}_0 \gamma_i\delta_i^\top
	+ \sum_{j=1}^q \big( \hat{c}_0^{(j)} \gamma_i\delta_i^\top \alpha_i(\beta_i^\top\alpha_i)^{j-1}\beta_i^\top + \hat{c}_j^{(j)} \alpha_i(\beta_i^\top\alpha_i)^{j-1}\beta_i^\top \gamma_i\delta_i^\top \big) \\
	&\quad + \sum_{j=1}^q \alpha_i (\beta_i^\top\alpha_i)^{j-1}\beta_i^\top \gamma_i \delta_i^\top \alpha_i P_j \beta_i^\top, \quad\mbox{with } P_j:=\sum_{k=0}^{q-j-1}  \hat{c}_j^{(k+j+1)} (\beta_i^\top\alpha_i)^{k}.
\end{aligned}
\end{equation*}
The terms:
\begin{itemize}
\item $\{P_j\}_{j=1}^q$ can be computed in $O(\Nfh r_i^2 + q^2 r_i^3)$ operations;\\[-1.5ex]
\item $\{(\beta_i^\top\alpha_i)^{j}\}_{j=1}^q$ can be computed in $O(q r_i^3)$ operations;\\[-1.5ex]
\item $\{\beta_i^\top \gamma_i \delta_i^\top \alpha_i P_j\}_{j=1}^q$ can be computed in $O(\Nrh q r_i^2 + \Nfh \Nrh r_i)$ operations.
\end{itemize}
Therefore, the overall computational cost to evaluate each $\widehat{\Lambda}_i$ is
$O(\Nfh r_i^2 + q^2 r_i^3 + \Nrh q r_i^2 + \Nfh \Nrh r_i)$.
Using $r_i=\rank{\mmi} \leq 4\Nrh (i-1)(q+1)$ and summing over the stages of the RK scheme, all terms involved in Algorithm~\ref{algo:RKMK} can be evaluated with arithmetic complexity $O(\Nfh \Nrh^2 q^2 \Ns^3 + \Nrh^3 q^5 \Ns^4) + C_{\Fcal}$,
where $C_{\Fcal}$ is the
complexity of the algorithm to compute $\Fcal(U)$ in \eqref{eq:Fcal} at any given $U\in\Mcal$.
The latter is, thus, the computational complexity of Algorithm~\ref{algo:RKMK} with $\coo=\exp$.
Since each $\mmi$ can be written as the sum of elements in the Lie algebra $\fr{g}_{\Nf}$, namely 
$\mmi=\sum_{j=1}^{i-1}\alpha_j\beta_j^\top$ with $\alpha_j,\beta_j\in\R{\Nf}{r_i}$,
one might suggest to approximate $\exp(\mmi)$
with
$E(\mmi):=\Pi_{j=1}^{i-1}\exp(\alpha_j\beta_j^\top)$ in the spirit of \cite{CI00}.
However, such an approximation does not bring significant computational savings nor it is guaranteed to provide a good approximation of the exponential map.}

\appendix
\section{Efficient computation of the inverse tangent map}\label{app:RtangInv}
We propose an algorithm to solve \eqref{eq:invR} with a computational cost of order $O(\Nfh\Nrh^2)$.
We proceed exactly as in \cite[Section 3.2.1]{CO03} with the only difference that
we consider any arbitrary $S\in\sym(\Nr)\cap\fr{sp}(\Nr)$. 

Using the definition of the derivative of the Cayley transform \eqref{eq:dcayinv}
we can recast \eqref{eq:invR} as
\begin{equation}\label{eq:eqTM}
	\aqvt(\Idm_{\Nf}+\aqvb)^{-1}\rqv - (\Idm_{\Nf}-\aqvb)W=0,
	\qquad \aqvb:=\dfrac{\aqv}{2}.
\end{equation}
Moreover, using the definition of $\rqv$ in \eqref{eq:retraction} results in
\begin{equation*}
\begin{aligned}
	2\rqv & = \big(\Idm_{\Nf}+\aqvb\big)\rqv + \big(\Idm_{\Nf}-\aqvb\big)\rqv\\
	 	 & = \big(\Idm_{\Nf}+\aqvb\big)\rqv +
	 	 	 \big(\Idm_{\Nf}-\aqvb\big)\big(\Idm_{\Nf}-\aqvb\big)^{-1}\big(\Idm_{\Nf}+\aqvb\big)Q\\
	 	 & = \big(\Idm_{\Nf}+\aqvb\big)\big(\rqv+Q\big).
\end{aligned}
\end{equation*}
Therefore, substituting in \eqref{eq:eqTM} and using the definition of $\Upsilon_Q$ from \eqref{eq:defA} gives
\begin{equation}\label{eq:eqgsv}
	\gsv Q^\top(\rqv+Q) - Q\gsv^\top (\rqv+Q) - \left(2\Idm_{\Nf}-\aqv\right)W=0.
\end{equation}
We proceed by solving problem \eqref{eq:eqgsv} for $\gt:=\gsv\in\T{Q}{\Sp(\Nr,\r{\Nf})}$ and then,
in view of \eqref{eq:idF}, we recover
$\vt\in\bl{\TsM{Q}}$ as $\vt=\gt-Q\gt^\top Q$,
at a computational cost of order $O(\Nfh\Nrh^2)$.

It is possible to
recast problem \eqref{eq:eqgsv} as $\gsv = QT_1(\vt)+T_2$, where
\begin{equation*}
\begin{aligned}
	T_1(\vt) & := \gsv^\top (\rqv+Q)(Q^\top\rqv+\Idm_{\Nr})^{-1},\\
	T_2 & := \left(2\Idm_{\Nf}-\aqv\right)W(Q^\top\rqv+\Idm_{\Nr})^{-1}.
\end{aligned}
\end{equation*}
The term $T_2$, independent of $\vt$, can be computed in $O(\Nfh\Nrh^2+\Nrh^3)$ operations.
Indeed, since $\aqv=\alpha\beta^\top$ as defined in \eqref{eq:AB},
the term $\aqv W$ can be computed as $\alpha(\beta^\top W)$ in $O(\Nfh\Nrh^2)$ flops.
The term $T_1(\vt)$ can be expressed as $T_1(\vt) = Q^\top\gsv+Q T_2$.
Using the fact that
$Q^\top\gsv + \gsv^\top Q=2S$,
the symmetric part of $T_1$ reads $T_1+T_1^\top = 2S-Q^\top T_2+T_2^\top Q$.
Moreover,
\begin{equation*}
\begin{aligned}
	& (\rqv+Q)^\top \gsv = (\rqv+Q)^\top T_2 + (\rqv^\top Q+\Idm_{\Nr})T_1(\vt),\\
	& (\rqv^\top Q+\Idm_{\Nr})T_1(\vt)^\top = (Q^\top\rqv+\Idm_{\Nr})^\top(Q^\top\rqv+\Idm_{\Nr})^{-\top}(\rqv+Q)^\top \gsv.
\end{aligned}
\end{equation*}
The skew-symmetric part of $T_1$ is then
$T_1-T_1^\top = -(\rqv^\top Q+\Idm_{\Nr})^{-1}(\rqv+ Q)^\top T_2$.
Therefore,
\begin{equation*}
\begin{aligned}
	2 T_1(\vt) & = \big((T_1(\vt)+T_1(\vt)^\top)+(T_1(\vt)-T_1(\vt)^\top)\big)\\
		& = 2S-(Q^\top T_2-T_2^\top Q)-(\rqv^\top Q+\Idm_{\Nr})^{-1}(\rqv+ Q)^\top T_2.
\end{aligned}
\end{equation*}
It is straightforward to show that all operations involved in the computation of $T_1$ can be done
with complexity of order $O(\Nfh\Nrh^2)$.

\end{document}